\documentclass{article}  
\usepackage{graphicx}
\usepackage{amsmath}
\usepackage{amsthm}
\usepackage[T1]{fontenc}
\usepackage{amssymb}
\usepackage[utf8]{inputenc}
\usepackage{graphicx}
\usepackage{indentfirst}
\usepackage[hyphens]{url} 
\usepackage[english]{babel}	
\usepackage{hyperref}			
\usepackage{color}
\usepackage{float} 
\usepackage{tikz}
\usetikzlibrary{arrows}
\usepackage{caption}
\usepackage{subcaption}
\usepackage{babel} 
\usepackage{wrapfig}
\usepackage{envmath}
\numberwithin{equation}{section}
\usepackage[justification=centering]{caption} 
\usepackage{pgfplots} 

\addto\captionsfrench{}

\def\Hkk{\mathrm{H}^{k}}   
\def\Hk1{\mathrm{H}^{k+1}}
\def\HH1{\mathrm{H}^1}
\def\Hexpo{\mathrm{H}^}
\def\L2{\mathrm{L}^2}  
\def\LL{\mathrm{L}^}  
\def\Linf{\mathrm{L}^\infty} 
\def\H2{\mathrm{H}^2}

\def\c{\mathcal{C}^} 
\def\P{\mathbb{P}^}
\def\omgam{(\Omega,\Gamma)}

\def\Vh{\mathbb{V}_h}

\def\Vhlifte{\Vh^\ell}

\def\d{\mathrm{d}}

\def\nn{\boldsymbol{\mathrm{n}}}

\def\nt{\nabla_\Gamma}
\def\na{\nabla}
\def\div{\mathrm{div}}
\def\dist{\mathrm{dist}}
\def\diff{\mathrm{D}}
\def\ds{\mathrm{ds}}

\def\omhh{\Omega_h}
\def\omh1{\Omega_h^{(1)}}
\def\omhr{\Omega_h^{(r)}}

\def\ghh{\Gamma_h}
\def\gh1{\Gamma_h^{(1)}}
\def\ghr{\Gamma_h^{(r)}}

\def\ft{F_T}
\def\fte{F_T^{(e)}}
\def\ftr{F_T^{(r)}}
\def\ftre{F_{T^{(r)}}^{(e)}}

\def\Ghr{G_h^{(r)}}
\def\G{\mathcal{G}_h^{(r)}} 

\def\Ahell{A_h^\ell}
\def\ahell{a_h^\ell}
\def\mhell{m_h^\ell}

\def\Jb{J_b}
\def\Jh{J_h}
\def\Jblifte{J_b^\ell}
\def\Jhlifte{J_h^\ell}

\def\tref{\hat{T}}
\def\trefminissigma{\tref\backslash\hat{\sigma}}
\def\hatsigma{\hat{\sigma}}

\def\te{{T}^{(e)}}

\def\tr{{T}^{(r)}}

\def\tauh{\mathcal{T}_h^{(1)}}

\def\taur{\mathcal{T}_h^{(r)}}
\def\taue{\mathcal{T}_h^{(e)}}

\def\lambdaetoile{\lambda^*}

\def\hatx{\hat{x}}
\def\haty{\hat{y}}
\def\hatv{\hat{v}}

\newcommand{\fonction}[5]{\begin{array}[t]{lrcl}#1 :&#2 &\longrightarrow &#3\\&#4& \longmapsto &#5 \end{array}}

\newtheorem{coro}{Corollary}

\newtheorem{proposition}{Proposition}
\newtheorem{definition}{Definition}
\newtheorem{remarque}{Remark}

\newtheorem{theoreme}{Theorem }
\newtheorem{lem}{Lemma}

\numberwithin{proposition}{section}
\numberwithin{definition}{section}
\numberwithin{coro}{section}
\numberwithin{Propriete}{section}
\numberwithin{remarque}{section}
\numberwithin{ex}{section}
\numberwithin{theoreme}{section}
\numberwithin{lem}{section}
\numberwithin{notation}{section}
\numberwithin{hypothese}{section}

\addto\captionsfrench{}

 \def\projG{\Pi_h}

\def\ProgZj{\mathcal{P}_{\ahell}}

\def\ProjHj{\mathcal{P}_{\mhell}}

\def\projPj{\mathcal{P}_{m}}

\def\F{\mathbb{E}_{\lbdi}}

\def\Fh{\mathbb{F}_h} 
\def\Fhlift{\mathbb{F}_h^\ell}
\def\Shlift{\mathbb{S}_h^\ell}

\def\Ehj{\mathbb{E}_{\lbdhj}} 
\def\Ehjlift{\Ehj^\ell}

\def\eL2{e_{\L2}}
\def\eH1{e_{\HH1_0}}
\def\elbd{e_{\lambda_6}}
\def\elbdd{e_{\lambda_{10}}}

\def\uj{u_j}
\def\up{u_p}

\def\uhj{U_j}
\def\uhjlifte{U_j^\ell}

\def\uhp{U_p}
\def\uhplifte{\uhp^\ell}

\def\lbdi{\lambda_i}
\def\lbdj{\lambda_j}
\def\lbdhj{\Lambda_j}
\def\lbdhp{\Lambda_p}
\def\lbdtildej{\tilde{\Lambda}_j}
\def\lbdtildep{\tilde{\Lambda}_p}

\def\ci{c_{\lbdi}}

\def\Ji{\mathrm{J}}
\def\J{\mathrm{J}}

\def\muJ{\mu_J}

\def\R{\mathbb{R}}      
\def\N{\mathbb{N}}  
\def\I{\mathrm{Id}} 
\def\transpose{^\mathsf{T}}

\def\Ih1{\mathcal{I}^{(1)}}
\def\Ihr{\mathcal{I}^{(r)}}
\def\Ihlifte{\mathcal{I}^\ell}

\title{Finite element analysis of a spectral problem\\on curved meshes occurring in diffusion\\with high order boundary conditions}
\author{Fabien Caubet\footnotemark[1]\footnote{Universit\'e de Pau et des Pays de l'Adour, E2S UPPA, CNRS, LMAP, UMR 5142, 64000 Pau, France. \texttt{fabien.caubet@univ-pau.fr}, \texttt{joyce.ghantous@univ-pau.fr}, \texttt{charles.pierre@univ-pau.fr}}, 
Joyce Ghantous\footnotemark[1], 
 Charles Pierre\footnotemark[1]
}

\begin{document}
\maketitle

\begin{abstract}
In this work is considered a spectral problem, involving a second order term on the domain boundary: the Laplace-Beltrami operator. A variational formulation is presented, leading to a finite element discretization. For the Laplace-Beltrami operator to make sense on the boundary, the domain is smooth: consequently the computational domain (classically a polygonal domain) will not match the physical one. Thus, the physical domain is discretized using high order curved meshes so as to reduce the \textit{geometric error}. The \textit{lift operator}, which is aimed to transform a function defined on the mesh domain into a function defined on the physical one, is recalled. This \textit{lift} is a key ingredient in estimating errors on eigenvalues and eigenfunctions. A bootstrap method is used to prove the error estimates, which are expressed both in terms of \textit{finite element approximation error} and of \textit{geometric error}, respectively associated to the finite element degree $k\ge 1$ and to the mesh order~$r\ge 1$. Numerical experiments are led on various smooth domains in 2D and 3D, which allow us to validate the presented theoretical results.
\end{abstract}


\medskip


\section{Introduction}
\paragraph{Motivation ad objective.} 
This work is part of a research program on the study of vibrating properties of mechanical parts submitted to intense and varying rotating regimes, more specifically when these parts are surrounded by thin surface layers (either specific industrial treatments or corrosion) that may impact or alter their mechanical properties.
The general objective is to use shape optimization to better understand such mechanical parts and improve their design.
This work is part of the RODAM research project\footnote{\textit{Robust Optimal Design under Additive Manufacturing constraints}: \url{https://lma-umr5142.univ-pau.fr/en/scientific-activities/scientific-challenges/rodam.html}}.

\medskip

From the numerical point of view, taking into account thin layers around mechanical parts induces specific difficulties, in particular when discretizing the domain with a mesh size adapted to the thin layer.
To overcome this issue, the thin layer is modeled by adapted boundary conditions involving second order terms such as the Laplace-Beltrami operator. 
These conditions derive from the pioneering works of Ventcel in \cite{Ventcel-1956,Ventcel-1959}. 
For the second order boundary terms to make sense, the domain is assumed to be smooth.
Thus, we have to deal with problems where the physical domain and the mesh domain differ, putting forward an intrinsic geometric error. 

\medskip

Towards this general objective of optimizing spectral properties for the elastic behavior of mechanical parts surrounded by specific thin layers, the present paper addresses the numerical resolution of the direct problem. The original problem is first simplified considering in a first step a scalar diffusion problem instead of the vector linear elasticity framework. Our aim is to analyze the numerical errors when solving the spectral problems both when considering the error induced by the numerical method and the geometric error caused when discretizing the domain. 

\medskip

This paper thus is devoted to the numerical analysis of a spectral elliptic problem equipped with a non classical boundary condition involving a high order tangential operator, here the Laplace-Beltrami operator: 
the so-called \textit{Ventcel boundary condition} (we refer, e.g., to \cite{ref-ventcel} for a general derivation of these boundary conditions). 




\paragraph{The Ventcel eigenvalue problem.} 
Let  $\Omega$ be a nonempty bounded connected smooth domain of~$\R^{d}$, $d=2, 3$, with a smooth boundary $\Gamma := \partial\Omega$. Motivated by generalized impedance boundary conditions, we consider the following spectral problem with Ventcel boundary conditions,
\begin{equation}
\label{sys-eigenval-ventcel}
\arraycolsep=2pt
    \left\lbrace
\begin{array}{rcll}
-\Delta u &=& \lambda u    & \text{ in } \Omega, \\
 - \Delta_{\Gamma} u + \partial_{\mathrm{n}} u + u &=&0 & \text{ on } \Gamma,
\end{array}
\right.
\end{equation}
where $\partial_{\mathrm{n}} u$ is the normal derivative of $u$ along~$\Gamma$ and~$\Delta_\Gamma$ the Laplace-Beltrami operator (see below for a reminder of the definition). The operator associated to this problem is a self-adjoint positive-definite operator. Consequently, the spectrum of Problem \eqref{sys-eigenval-ventcel} consists of an increasing sequence of positive eigenvalues tending to infinity. For each eigenvalue there exists a finite number of associated eigenfunctions defined on $\Omega$. In \cite{ventcel1}, 
the authors delivered a thorough study of the well-posedness of the Ventcel problem with source terms and the regularity of its solution.

\medskip

As previously mentioned, the domain $\Omega$ is required to be smooth due to the presence of second order boundary conditions. 
Therefore the physical domain $\Omega$ 
cannot match the mesh domain $\Omega_h$ inducing an intrinsic geometric error. This error is larger when using classical affine meshes made of triangles in 2D and tetrahedrals in 3D. Indeed, these polygonal meshes will induce a saturation of the numerical error at a low order: when resorting to accurate finite element methods the geometric error will dominate. To improve the error rate, we will resort to curved meshes whose elements have polynomial degree $r \ge 1$ so as to have a geometric error with a better asymptotic regime with respect to the mesh size. \medskip

A technical difficulty arises: the mesh domain of order~$r$, denoted~$\omhr$, does not fit exactly on~$\Omega$. A $\P k$-Lagrangian finite element method is used with a degree~$k \ge 1$ to approximate the exact solutions of System~\eqref{sys-eigenval-ventcel}. In order to estimate the error between the discrete eigenfunctions defined on $\omhr$ and the exact ones defined on $\Omega$, a \textit{lift operator} is required. Throughout the years, many authors defined their version of the \textit{lift functional}, like in~\cite{dubois,Lenoir1986,nedelec,D1,D2,Dz88}. Among them, Dubois defined a lift based on the orthogonal projection onto the domain boundary $\Gamma$ in~\cite{dubois}. The idea of relying on the orthogonal projection in the lift definition was reintroduced by Dzuik in \cite{Dz88} in order to define a surface lift. This was generalised in the case of \textit{lifting} a function from higher order surface meshes onto a continuous surface in~\cite{D1} by Demlow. In the context of curved meshes with order $r\ge 2$, the definition of a volume lift required a higher regularity: such an improvement was brought by Elliott \textit{et al.} in~\cite{elliott}, with a definition that also relies on the orthogonal projection. However, their definition of the volume lift did not fit the orthogonal projection on the computational domain's boundary, as we highlighted in~\cite{art-joyce-1}. It is natural to expect the volume lift to fit the orthogonal projection on the mesh boundary: such a property is crucial for the derivation of \textit{a priori} error estimates in~\cite{art-joyce-1} and in the present paper.
Thus, we recently formulated an alternative definition of a volume lift in~\cite{art-joyce-1} to satisfy that property together with all the necessary regularity properties, which will be adopted and recalled in this paper.

\paragraph{State of the art and main results.} 
In 2013, Elliott and Ranner in \cite{elliott} 
made a numerical analysis of a bulk problem 
with a Ventcel boundary condition on curved meshes 
with iso-parametric finite elements. 
More recently the  approximation of the Ventcel problem with source terms has been studied in  \cite{ed,art-joyce-1} with curved meshes and high order finite elements. In \cite{elliott,ed}, the same lift definition was considered fulfilling its role in enabling the authors to estimate the error on a smooth domain while using an isoparametric approach. In \cite{art-joyce-1,Jaca}, a non isoparametric approach is led while distinguishing between the mesh order $r$ and the degree of the finite element method~$k$, using a different definition of the lift operator as previously discussed. This lift definition is considered in this work and recalled in Section~\ref{sec:mesh}. 

\medskip

Meanwhile, in 2018, error estimates for spectral problem on curved meshes have been carried out in \cite{D4} by Bonito \textit{et al.}for the surface Laplacian. The ideas of~\cite{D4} are adapted and extended in the present paper in the case of a volume spectral problem. The main novelties is the use of the new lift operator defined in~\cite{art-joyce-1} to estimate the eigenvalue and eigenfunction error both in terms of finite element approximation error and of geometric error, respectively, associated to the finite element degree~$k \ge 1$ and to the mesh order~$r \ge 1$. To the authors' knowledge, no error analysis was made on a bulk spectral problem having Ventcel type conditions. Let us also emphasize that the theoretical study and the numerical resolution of this spectral problem involve non-trivial difficulties compared with the analysis of the direct problem we made in~\cite{art-joyce-1}.

\medskip

The main result of this paper can be summarized as follows (see Theorem~\ref{th-error-bound-eigenfunction} for a precise statement). Let $\lbdi$ be an eigenvalue of multiplicity~$N$ with its corresponding eigenfunctions,~$\{ \uj \}_{j \in \Ji}$, where $\Ji:=\{i,...,i+N-1\}$, relatively to Problem~\eqref{sys-eigenval-ventcel}. Then, there exists a mesh independent constant $\ci > 0$, such that, for any~$j\in \J$,
\begin{eqnarray*}
    |\lbdj- \lbdhj  | &\le& \ci ( h^{2k} + h^{r+1}),\\[0.15cm]
    \inf_{U \in \Fhlift}\|\uj -U \ \|_{\L2 (\Omega)}  &\le& \ci ( h^{k+1} + h^{r+1/2}), \\
    \inf_{U\in \Fhlift} \|\uj - U \|_{\HH1\omgam } &\le& \ci ( h^{k} + h^{r+1/2}),
\end{eqnarray*}
where $\lbdhj$ is the  eigenvalue of the discretization 
of \eqref{sys-eigenval-ventcel}
of rank $j\in \Ji$,
$\Fh$ is the space generated by the discrete eigenfunctions associated to $\{\lbdhj\}_{j \in \J}$, $\Fhlift$ is the lift of $\Fh$  made of functions defined on the physical domain $\Omega$ and where $h$ is the mesh size. The Hilbert spaces~$\HH1 \omgam$, precisely defined below, is made of the functions in 
${\rm H}^1(\Omega)$ having their traces in ${\rm H}^1(\Gamma)$. 

\medskip

A \textit{bootstrap method} is used to prove these error estimates. To sum up the main ideas of the proof, a preliminary estimation of the eigenvalue error is needed in order to estimate the eigenfunction error in the $\L2(\Omega)$ and $\HH1\omgam$ norms using orthogonal projections over the space~$\Fhlift$. Lastly, we are able to obtain the adequate eigenvalue error estimate with respect to the finite element degree~$k$ and the geometric order of the mesh~$r$ using the obtained estimates on the eigenfunctions. 

\medskip

We validate these estimations in several numerical experiments presented in two and three dimensions. We noticed a \textit{super-convergence} of the error rate on the quadratic meshes, which is much better than expected, as it was also depicted in~\cite{D4}. In an attempt to understand the origin of this phenomena, numerical experiments are led on a non-symmetric, non convex domain and also on classical domains like the unit disk and the unit ball. However, the same asymptotic regime of the errors is observed on these various smooth domains.

\paragraph{Paper organization.} Section \ref{sec:notations_def} contains all the mathematical tools and useful definitions to derive the weak formulation of System~\eqref{sys-eigenval-ventcel}. Section~\ref{sec:mesh} is devoted to the definition of the high order meshes and the lift operator with some of its most essential properties. A Lagrangian finite element space and a discrete formulation of System~\eqref{sys-eigenval-ventcel} are presented in Section~\ref{sec:FEM}, alongside their \textit{lifted forms} onto $\Omega$. In Section~\ref{sec:error-estimation} is stated the main result and are detailed the proofs of the eigenvalue and eigenfunction estimates. The paper wraps up in Section~\ref{sec:numerical-ex} with numerical experiments studying the convergence rates of the eigenvalue and eigenfunction errors  over various domains in 2D and 3D.

\section{The continuous problem}
\label{sec:notations_def}
\paragraph{Needed mathematical tools.}
%
%
%
%
%
%
%
%
%
Firstly, let us introduce the notations that we adopt in this paper. Throughout this paper, $\Omega$ is a nonempty bounded connected open subset of $\R^{d}$ $(d=2,3)$ with a smooth (at least $\c2$) boundary~$\Gamma:=\partial{\Omega}$. 
The unit normal to~$\Gamma$ pointing outwards is denoted by~$\nn$ and $\partial_{\mathrm{n}} u$ is the normal derivative of a function~$u$.
We denote respectively by $\LL 2(\Omega)$ 
and  $\LL 2(\Gamma)$ 
the usual Lebesgue spaces endowed with their standard norms on $\Omega$ and $\Gamma$.
Moreover, for~$k \geq 1$,~$ {\Hkk(\Omega)}$ denotes the usual Sobolev space endowed with its standard norm. We also consider the Sobolev spaces  {$\Hkk(\Gamma)$} on the boundary as defined e.g. in \cite[\S 2.3]{ventcel1}. It is recalled that the norm on $\Hexpo{1}(\Gamma)$ is 
:~$\|u\|^2_{\Hexpo{1}(\Gamma)} : =  \|u\|^2_{\L2(\Gamma)} + \|\nt u\|^2_{\L2(\Gamma)},$ where $\nt$ is the tangential gradient defined below; and that $ {\|u\|^2_{\Hkk(\Gamma)} :=  \|u\|^2_{\Hexpo{k-1}(\Gamma)} + \|\nt u\|^2_{\Hexpo{k-1}(\Gamma)}}$. 
Throughout this work, we rely on the following Hilbert space (see \cite{ventcel1}) 
$$
    \Hexpo{1}\omgam := \{ u \in \Hexpo{1}(\Omega), \ u_{|_\Gamma} \in \Hexpo{1}(\Gamma) \},
$$
equipped with the norm $\|u\|^2_{\Hexpo{1}\omgam} : =  \|u\|^2_{\Hexpo{1}(\Omega)} + \| u\|^2_{\Hexpo{1}(\Gamma)}.$
In a similar way is defined the following space 
$\L2 \omgam := \{ u \in \L2(\Omega), \ u_{|_\Gamma} \in \L2(\Gamma) \},$
equipped with the norm~$\|u\|^2_{\L2\omgam} :=  \|u\|^2_{\L2(\Omega)} + \| u\|^2_{\L2(\Gamma)}$. 
 More generally, we define~$ {\Hkk \omgam:=\{ u \in \Hexpo{k}(\Omega), \ u_{|_\Gamma} \in \Hexpo{k}(\Gamma) \}}$.  

\medskip

Secondly, to understand more the so-called Ventcel boundary conditions, we recall the definition of the Laplace-Beltrami operator (see~\cite{livreopt}): 
the {\it Laplace-Beltrami} operator of $u \in \H2 (\Gamma)$ is given by,
\begin{equation*}
    \Delta_\Gamma u  := \div_\Gamma (\nt u),
\end{equation*}
where, 
\begin{itemize}
\item the {\it tangential gradient} of 
  $u$ is given by $\nt u :=\na \tilde{u}  - (\na \tilde{w} \cdot \nn)\nn$, with~$\tilde{u} \in \HH1(\R^d)$ being any extension of $u$;
\item the {\it tangential divergence} of $W\in\HH1(\Gamma,\R^d)$ is  
  $\div_{\Gamma} W : = 
  \div \tilde{W}- (\mathrm{D}\tilde{W}\, \nn)\cdot \nn$, 
  where~$\tilde{W} \in \HH1(\R^d, \mathbb{R}^d)$ 
  is any extension of $W$ and $\mathrm{D}\tilde{W}$ is the differential of~$\tilde{W}$;
\end{itemize}

Finally, the construction of the mesh used in Section \ref{sec:mesh} is based on the following fundamental result that may be found in \cite{tubneig} and \cite[\S 14.6]{GT98}. For more details on the geometrical properties of the tubular neighborhood and the orthogonal projection defined below, we refer to \cite{D1,D2,actanum}.
%
%
%
%
%
%
%
%
\begin{proposition}
\label{tub_neigh_orth_proj_prop}
Let $\Omega$ be a nonempty bounded connected open subset of $\R^{d}$ $(d~=~2,3)$ with a~$\c2$ boundary $\Gamma= \partial \Omega$. Let $\d : \R^d \to \R$ be the signed distance function with respect to $\Gamma$ defined by,
\begin{equation*}
  \d(x) :=
  \left\lbrace
    \begin{array}{ll}
      -\dist(x, \Gamma)&  {\rm if } \, x \in \Omega ,
      \\
      0&  {\rm if } \, x \in \Gamma ,
      \\
      \dist(x, \Gamma)&  {\rm otherwise},
    \end{array}
  \right. \qquad 
  {\rm with} \quad \dist(x, \Gamma) := \inf \{|x-y|,~ \ y \in \Gamma \}.
\end{equation*}
Then there exists a tubular neighborhood $\mathcal{U}_{\Gamma}:= \{ x \in \R^d ; |\d(x)| < \delta_\Gamma \}$ of $\Gamma$, of sufficiently small width $\delta_\Gamma$, where {$\d$ is a $\c2$ function}. Its gradient 
$\na \d$ is an extension of the external unit normal $\nn$ to $\Gamma$. Additionally, in this neighborhood~$\mathcal{U}_{\Gamma}$, the orthogonal projection~$b$ onto $\Gamma$ is uniquely defined
and given by,
\begin{displaymath} 
b\, :~ x \in \mathcal{U}_{\Gamma}  
\longmapsto    b(x):=x-\d(x)\na \d(x) \in \Gamma.
\end{displaymath}
\end{proposition}
%
%
%
%
%
%
%
\paragraph{Well posedness of the spectral problem.}
%
%
%
Throughout the rest of the paper, $\d x$ and $\ds$ denote respectively the volume and surface measures on $\Omega$ and on $\Gamma$.

The variational formulation of the studied problem \eqref{sys-eigenval-ventcel} is classically obtained, using the integration by parts formula on the surface $\Gamma$ (see, e.g., \cite{livreopt}): it is then given by, 
\begin{equation}
\label{fv_faible}
  \left\lbrace
      \begin{array}{l}
       \mbox{find } (\lambda, u) \in \R \times \HH1 \omgam, \ \mbox{ such that,}  \\
      a(u,v) = \lambda \, m(u,v), \quad  \forall \ v \in \HH1\omgam,
      \end{array}
    \right.
\end{equation}
where $a$ is the bilinear form, defined on $[\HH1\omgam]^2$, given by,
\begin{equation*}
    a(u,v) :=  \int_{\Omega} \nabla u \cdot \nabla v \, \d x + \int_{\Gamma} \nabla_{\Gamma} u \cdot \nabla_{\Gamma} v \, \ds + \int_{\Gamma} u  v \, \ds ,
\end{equation*} 
and $m$ is the bilinear form, defined on $[\HH1\omgam]^2$, given by,
\begin{equation*}
     m(u,v) := \int_{\Omega} u  v \, \d x.
\end{equation*}

The bilinear form $a$, being symmetric and continuous, is also coercive with respect to the norm over $\HH1 \omgam$, as proved in \cite[Theorem 2]{Jaca}. The second bilinear form $m$ is none other than the scalar product on the space $\L2(\Omega)$. Then by a classical spectral result (see \cite[Theorem~7.3.2]{GA}), we claim the existence of an infinite number of eigenvalues to Problem~\eqref{fv_faible}, which form an increasing sequence $(\lambda_n)_{n\ge 1} \subset {\R }^{*}_{+}$ of positive real numbers, tending to infinity. Their associated eigenfunctions form an orthonormal Hilbert basis of $\L2 (\Omega)$, denoted $(u_n)_{n\ge 1}$ satisfying,
\begin{equation*}
    u_n \in \HH1\omgam, \  \ a(u_n,v) = \lambda_n m( u_n,v ),\ \ \forall \ v \ \in \  \HH1\omgam.
\end{equation*}

Assuming that the eigenvalues are counted with their multiplicity and ordered increasingly, the aim of this work is to approximate an eigenvalue $\lbdi$ of multiplicity $N \ge 1$ 
%
%
%
and  its associated eigenfunctions $\{ \uj \}_{j \in \Ji}$, which are $m$-orthonormal, with  $\Ji= \{i,...,i+N-1 \}$ the set of indices (see~\cite[page 5]{23}). 
\section{Curved mesh and lift definition}
\label{sec:mesh}
Throughout this section, we briefly explain the construction of curved meshes of geometrical order~$r\ge 1$ of the domain~$\Omega$ and give the main associated notations. We refer to Appendix~\ref{mesh:appendix} for details and rigorous
definitions (in particular concerning the mentioned
transformations). Additionally, the definition of the lift operator with some essential lift properties are recalled. We refer to~\cite[\S 3, \S 4]{art-joyce-1} for more exhaustive details and properties. The set of polynomials in~$\R^d$ of order~$r$ or less is denoted by $\P r$. 
From now on, the domain~$\Omega$, is assumed to be at least $\c {r+2}$ regular, and~$\tref$ denotes the reference simplex of dimension~$d$. \medskip
%
%
%
%
%
%

Let $\tauh$ be a polyhedral quasi-uniformal mesh of $\Omega$ made of simplices of dimension $d$, denoted~$T$ (triangles or tetrahedra). The mesh domain $\omh1:=\{\cup \, T, \, T \in \tauh\} $ does not coincide with the physical domain $\Omega$, which is smooth. Each mesh element $T$ is the image of the reference simplex~$\tref$ by the affine transformation $\ft:\tref \to T$. An exact mesh $\taue$ (with domain $\Omega$) is built. To each element $T\in \tauh$ is associated a transformation~$\fte:\tref \to \te:=\fte(\tref)$, which is defined with the help of $\ft$ in Definition \ref{def:fte-y}. The elements of the exact mesh $\taue$ exactly are~$\{ \te=\fte(\tref), \, T \in \tauh\} $, where the exact elements~$\te$ share the same vertices as $T$. All the details are given in  Appendix~\ref{mesh:appendix}.

\paragraph{Curved mesh $\taur$ of order $r$.}
The exact mapping $\fte$, defined in Appendix~\ref{mesh:appendix}, is interpolated as a polynomial of order $r \ge 1$ in the classical $\P r$-Lagrange basis on $\tref$. The interpolant is denoted by $\ftr$, which is a $\c1$-diffeomorphism and is in $\c {r+1}(\tref)$ (see \cite[chap. 4.3]{PHcia}). For more exhaustive details and properties of this transformation, we refer to \cite{elliott,ciaravtransf,PHcia}. 
Note that, by definition,~$\ftr$ and~$\fte$ coincide on all $\P r$-Lagrange nodes in $\tref$. 
The curved mesh of order~$r$ is~$\taur := \{ \tr:= \ftr(\tref); \, T \in \tauh \}$, $\omhr := \cup_{\tr \in  \taur}\tr$ is the mesh domain and~$\ghr:= \partial \omhr$ is its boundary. 

\paragraph{Functional lift.}

In order to lift a function from the mesh domain onto $\Omega$, a well defined transformation going from $\omhr$ to $\Omega$ is needed. In a previous work~\cite{art-joyce-1}, the transformation $\Ghr$ was defined piece-wise such that, 
$$
    \Ghr : \omhr \to \Omega ; \quad {\Ghr}_{|_{\ghr}} = b,
$$
where~$b$ is the orthogonal projection defined in Proposition \ref{tub_neigh_orth_proj_prop}. We refer to Appendix \ref{lift} for the full expression of $\Ghr$. 

\medskip 

By construction, $\Ghr$ is globally continuous and piecewise differentiable on each mesh element. Additionally, quoting \cite[Proposition 2]{art-joyce-1}, where the full proof is detailed: let $\tr \in \taur$, the mapping ${\Ghr}_{|_{\tr}}$ is~$\c{r+1}(\tr)$ regular and a $\c 1$- diffeomorphism from $\tr$ onto $\te$.  Moreover, for a sufficiently small mesh size $h$, there exists a constant $c>0$, independent of $h$, such that, 
\begin{equation}
\label{ineq:Gh-Id_Jh-1} 
  \forall \ x \in \tr, \ \ \ \ \| \diff {\Ghr}(x) - \I \| \le c h^r \qquad \mbox{ and } \qquad 
    | \Jh(x)- 1 | \le c h^r.
\end{equation}
where $\diff \Ghr$ is the differential of $\Ghr$ and $\Jh$ is its Jacobin.

\begin{definition}
\label{def:liftvolume}
    To $u_h\in {\rm L}^2(\omhr,\ghr)$ is associated its lift, denoted $u_h^\ell \in {\rm L}^2\omgam$, given by, 
    $$
        u_h^\ell\circ \Ghr := u_h.
    $$
The lift satisfies the trace property which states
$$
  \forall ~ u_h \in {\rm H}^1(\omhr), \quad 
  \left ( {\rm Tr} ~u_h\right )^{\ell} = {\rm Tr} (u_h^\ell).
$$
\end{definition}

\begin{remarque}
The above trace property is essential in the error analysis detailed in Section~\ref{sec:error-estimation}. This is due to the fact that the restriction of $\Ghr$ to $\ghr$ is equal to the orthogonal projection: ${\Ghr}_{|_{\ghr}} = b.$
\end{remarque}

%
%
%
%
%
%
%
%
%
%
%
%
%
%
%
\paragraph{Lift properties.}
The results presented in this section can be found with more details in \cite{D1,D2,art-joyce-1}. 

%
%
%
%
%
Consider $u_h, v_h \in \HH1 (\omhh) $ and let $u_h^\ell, v_h^\ell \in \HH1 (\Omega)$ be their respected lifts, we have,
\begin{equation}
\label{pass_fct_scalaire_volume}
    \int_{\omhh}u_h v_h \, \d x = \int_\Omega u_h^\ell v_h^\ell \frac{1}{\Jhlifte} \d x,
\end{equation} 
where $\Jh$ denotes the Jacobian of $\Ghr$ and $\Jhlifte$ is its lift given by $\Jhlifte \circ \Ghr = \Jh$. 

Note that for any $x \in \omhr$, using a change of variables $z=\Ghr(x) \in \Omega$, one has, $(\nabla v_h)^\ell(z) = {\transpose} \diff {\Ghr}(x)  \nabla v_h^\ell{(z)},$ where $\transpose \diff \Ghr$ is the transpose of $\diff \Ghr$. 
Introducing the notation, $\G (z) := {\transpose}\diff {\Ghr}(x),$ one has,
\begin{equation}
\label{pass_grad_volume}
     \int_{\Omega^{(r)}_h} \nabla u_h \cdot \nabla v_h \, \d x = \int_{\Omega}  \G (\nabla u_h^\ell) \cdot \G (\nabla v_h^\ell) \frac{1}{\Jhlifte} \, \d x.
\end{equation}
A direct consequence of the inequalities \eqref{ineq:Gh-Id_Jh-1}, using the lift definition \ref{def:liftvolume}, is that both $\G$ and~$\Jhlifte$ are bounded on $\te$. Additionally, we have the following inequalities, which are a key ingredient for the proof of the error estimations,
    \begin{equation}
        \label{ineq:Ghr-Id_1/Jh-1}
        \forall \ x \in \te, \ \ \ \ \| \G (x) - \I \| \le c h^r \qquad \mbox{ and } \qquad
         \left| \frac{1}{\Jhlifte(x)}- 1 \right| \le c h^r.
    \end{equation}

Similarly, let $u_h, v_h \in \HH1 (\ghh) $ with $u_h^\ell, v_h^\ell \in \HH1 (\Gamma)$ as their respected lifts. Then, one has,
\begin{equation}
    \label{pass_fct_scalaire_surface}
    \int_{\ghr}   u_h v_h \, \ds = \int_{\Gamma}   u^{\ell }_h v^{\ell }_h  \frac{ 1}{\Jblifte} \, \ds,
\end{equation}
where $\Jb$ denotes the Jacobian of the orthogonal projection $b$ defined in Proposition~\ref{tub_neigh_orth_proj_prop}, and $\Jblifte$ is its lift given by $\Jblifte \circ b = \Jb$.

A similar equation can be written with tangential gradients, given by the following expression, 
\begin{equation}
\label{pass_grad_surface}
    \int_{\ghr} \nabla_{\ghr} u_h \cdot \nabla_{\ghr}  v_h \, \ds _h = \int_{\Gamma} \Ahell \nabla_{\Gamma} u^{\ell }_h \cdot \nabla_{\Gamma} v^{\ell }_h  \, \ds ,
\end{equation} 
where $\Ahell$ is the lift of the matrix $A_h$ defined in \cite{art-joyce-1,D1}.

We recall two important estimates proved in~\cite{D1} relative to $A_h$ and $\Jb$. There exists a constant~$c>0$, independent of $h$, such that,
\begin{equation}
  \label{ineq:AhJh}
    ||\Ahell -\I ||_{\Linf(\Gamma)} \le c h^{r+1} \qquad \mbox{ and } \qquad 
        \left\| 1-\frac{1}{\Jblifte} \right\|_{\Linf(\Gamma)} \le ch^{r+1}.
\end{equation}

%
%
%
%
%
%
%
%
%
%
%

%
%
%
%
%
\section{Finite element approximation}
\label{sec:FEM}
In this section, is presented the finite element approximation of problem~\eqref{sys-eigenval-ventcel} using a $\P k$-Lagrange finite element method. We refer to, e.g., \cite{EG,PHcia} and \cite[\S 5]{art-joyce-1} for more details on finite element methods. From now on, we denote $\omhh$ and $\ghh$ to refer to~$\omhr$ and~$\ghr$, for any geometrical order $r\ge 1$.
\paragraph{Discrete formulation.} Let~$k \geq 1$ that denotes the finite element degree. Given a curved mesh~$\taur$, the $\P k$-Lagrangian finite element space is given by,
$$
    \Vh := \{ \chi \in C^0(\omhh); \ \chi|_T= \hat{\chi} \circ (\ftr )^{-1} , \ \hat{\chi} \in \mathbb{P}^k(\hat{T}), \ \forall \ T \in \taur \}.
$$
The approximation problem is to find $(\Lambda, U) \in \R\times\Vh$ such that,
\begin{equation}
\label{fvd}
     a_h(U,V ) = \Lambda m_h(U,V ), \ \ \ \ \ \forall \ V  \in \Vh,
\end{equation}
where $a_h$ is the following bilinear form, defined on $\Vh \times \Vh$, 
$$
    a_h(U ,V ) := \int_{\omhh} \nabla U   \cdot \nabla V   \d x + \int_{\ghh} \nabla_{\ghh} U  \cdot \nabla_{\ghh} V  \ds _h +  \int_{\ghh} U   V   \ds _h,
$$
and $m_h$ is the following bilinear form, defined on $\Vh \times \Vh$, 
$$
    m_h(U ,V ) = \int_{\omhh} U   V  \d x.
$$
%
%
%
%
%
%
%

%
%
%
\begin{remarque}
The discrete problem \eqref{fvd} admits an increasing finite sequence of positive discrete eigenvalues $\lbdhj \in \R_+^*$. There exists a basis of $\Vh$ made of discrete eigenfunctions~$\{ U_j \}_{j=1}^{\dim(\Vh)}$, which are $m_h$-orthogonal (see \cite[Lemma 7.4.1]{GA}). 
\end{remarque}
\paragraph{Lifted discrete formulation.} The lifted finite element space is given by, $$\Vhlifte := \{ v_h^\ell, \ v_h \in \Vh \}.$$ We define the lifted bilinear form $\ahell$, defined on $\Vhlifte \times \Vhlifte$, throughout,
$$
a_h(U ,V ) = \ahell(U^\ell,V^\ell), \quad \forall \, U , V  \in \Vh.
$$  
By applying \eqref{pass_grad_volume}, \eqref{pass_grad_surface} and \eqref{pass_fct_scalaire_surface}, then the expression of $\ahell$ is given by,
$$
    \ahell(U^\ell,V^\ell) = \int_{\Omega}  \G (\nabla U ^\ell) \cdot \G (\nabla V ^\ell) \frac{\, \d x}{\Jhlifte}+ \int_{\Gamma} A^{\ell }_h \nabla_{\Gamma} u^{\ell }_h \cdot \nabla_{\Gamma} v^{\ell }_h  \, \ds  
     + \int_{\Gamma} U ^\ell V^\ell \frac{\, \ds }{\Jblifte}.
$$
%
%
%
%
%

In a similar way, using \eqref{pass_fct_scalaire_volume}, we define the expression of $\mhell$, defined on~$\Vhlifte \times \Vhlifte$, throughout~$m_h(U ,V) = \mhell(U ^\ell,V^\ell)$ for $U , V \in \Vh,$ as follows,
\begin{equation*}
    \mhell(U ^\ell,V^\ell) =  \int_{\Omega} U   V \frac{\d x}{J_h^\ell}= m_h(U ,V).
\end{equation*}

Thus, we define the lifted formulation of Problem~\eqref{fvd} given by: find~$(\Lambda, U ^\ell) \in \R \times \Vh^\ell$ such that, 
\begin{equation}
 \label{fvdlifte}
      \ahell(U ^\ell,V)= \Lambda \ \mhell(U ^\ell,V), \quad \forall \, V \in \Vhlifte.
\end{equation}
\begin{remarque}
 The lifted problem \eqref{fvdlifte} shares the same eigenvalues as the discrete problem \eqref{fvd}, denoted $\{ \lbdhj\}_{j=1}^{\dim (\Vh)}$, which are associated to the lift of the eigenfunctions of the discrete formulation \eqref{fv_faible}, denoted $\{\uhjlifte\}_{j=1}^{\dim (\Vh)}$. Throughout the rest of this work, we suppose that the discrete eigenfunctions satisfy that~$\| U_j^\ell\|_{\mhell} =1$, for all $j=1,\dots, \dim(\Vh)$.
\end{remarque}
%
%
%
%
%
%
%
%
%
%
%
%
%
\section{Error analysis}
\label{sec:error-estimation}
First of all, the exact eigenvalues are ordered increasingly with their multiplicities. Let $i \in \N^*$. The aim of this work is to estimate the error produced when approximating the eigenvalue~$\lbdi$ of multiplicity $N$ and its corresponding eigenfunctions, $\{ \uj \}_{j \in \Ji}$ where $\Ji=\{i,...,i+N-1\}$, using a~$\P k$ finite element method on a curved mesh $\omhh$ with a geometrical order $r \ge 1$. These estimations are given in the following theorem, which is proved in the following sub-sections.

To this end, denote $\{\lbdhj\}_{j=1}^{\dim (\Vh)}$ the set of all the discrete eigenvalue. Each eigenvalue $\lbdhj$ is associated to an eigenspace $\Ehjlift$ in $\Vhlifte$, which is the set of all the discrete eigenfunctions associated to $\lbdhj$. Let $\Fhlift := \oplus_{j \in \J} \Ehjlift$ be the space containing all the eigenspaces associated to $\{\lbdhj\}_{j \in \J}$.

\medskip

Throughout this section, $c$ refers to a positive constant independent of the mesh size~$h$ and~$\ci$ refers to a positive constant dependent of the eigenvalue~$\lbdi$ and independent of $h$. From now on, the domain~$\Omega$, is assumed to be at least~$\c {k+1}$ regular such that the exact eigenfunctions of Problem~\eqref{sys-eigenval-ventcel} are in~$\Hk1\omgam$.
\begin{theoreme}
\label{th-error-bound-eigenfunction}
Let $\lbdi$ be an eigenvalue of multiplicity $N$ with its corresponding eigenfunctions,~$\{ \up \}_{p \in \Ji}$ where $\Ji=\{i,...,i+N-1\}$, relatively to Problem~\eqref{fv_faible}.
Then, for any $j\in \J$, there exists $\ci > 0$, 
\begin{equation}
\label{err:eigenvalue}
    |\lbdj- \lbdhj  | \le \ci ( h^{2k} + h^{r+1}),
\end{equation}
where $\lbdhj$ is a discrete eigenvalue relatively to Problem \eqref{fvd}. Additionally, there exists $\ci > 0$ for any $j\in \Ji$ such that, 
\begin{equation}
\label{err:eigenvectors-L2}
    \inf_{U \in \Fhlift}\|\uj-U \ \|_{\L2 (\Omega)}  \le \ci ( h^{k+1} + h^{r+1/2}),
\end{equation}
\begin{equation}
\label{err:eigenvectors-H1}
    \inf_{U\in \Fhlift} \|\uj-U\|_{\HH1\omgam } \le \ci ( h^{k} + h^{r+1/2}),
\end{equation}
where $\Fhlift$ is the space
containing all the eigenspaces associated to $\{\lbdhj\}_{j \in \J}$.
\end{theoreme}
\begin{remarque}
    In a similar manner, there exists $\ci > 0$ such that, 
    $$
     \inf_{u \in \F}\|U -u\ \|_{\L2 (\Omega)}  \le \ci ( h^{k+1} + h^{r+1/2}), \quad \inf_{u \in \F} \|U -u\|_{\HH1\omgam } \le \ci ( h^{k} + h^{r+1/2}),
    $$
    where $U \in \Vhlifte$ is a discrete eigenfunction associated to $\lbdhj$ and $\F$ is the eigenspace of $\lbdi$. These estimations, which are analogous to those presented in Theorem \ref{th-error-bound-eigenfunction}, are a consequence of Lemma~\ref{lem-equiv-eigenvector-estimations} in Section~\ref{sec:error-estimation-eigenvalue-2}.
\end{remarque}
In order to prove this theorem, we will proceed in several steps. In a nutshell, the main steps of the proof are to first estimate the so-called geometric error, second calculate a preliminary eigenvalue estimation, third estimate the eigenfunction error, and finally combine the last two steps to improve the  eigenvalue error.

\subsection{Geometric error}%
Before estimating the geometric error, we need to present the following corollary, which plays a key role in the remaining of the article (see \cite[Lemma~2]{art-joyce-1}). 
%
%
%
\begin{coro}
\label{coro:h-1/2}
Let $v \in \HH1(\Omega)$ and $w \in \H2(\Omega)$, then, for a sufficiently small $h$, there exists $c>0$ such that the following inequalities hold,
\begin{equation}
\label{ineq:L2+h1/2}
    \|v\|_{\L2(B_h^\ell)} \le c h^{1/2} \|v\|_{\HH1(\Omega)},
\end{equation}
\begin{equation}
\label{ineq:H1+h1/2}
    \|\na w\|_{\L2(B_h^\ell)} \le c h^{1/2} \|w\|_{\H2(\Omega)},
\end{equation}
where $B_h^\ell= \{ \ \te \in \taue ; \  \te \ \mbox{has at least 2 verticies on } \Gamma \}.$
\end{coro}
We present the following property on the domain $B_h^\ell$, which plays a key role in the geometric error estimation,
\begin{equation}
    \label{eq:Jh-1=0=Dgh-Id}
    \frac{1}{\Jhlifte}-1 =0  \ \ \ \mbox{and} \ \ \ \G - Id = 0, \ \ \ \mbox{in} \ \Omega \setminus B_h^\ell.
\end{equation}

To estimate the geometric error produced while approximating a domain by a mesh of order $r \ge 1$, we bound the difference between the two bilinear forms $a$ and $\ahell$ (resp. $m$ and $\mhell$). 
\begin{proposition}
Let $v, w \in \Vh^\ell$. Then there exists $c>0$ such that,
\begin{equation}
\label{ineq:a-ahell}
    |(a-\ahell)(v,w)| \le c ( h^r  \| \nabla v\|_{\L2(B_h^\ell)} \| \nabla w\|_{\L2(B_h^\ell)} + h^{r+1}\|v\|_{{\HH1(\Gamma)}}\|w\|_{\HH1(\Gamma)}),
\end{equation}
\begin{equation}
\label{ineq:m-mhell}
    |(m-\mhell)(v,w)| \le c h^{r+1}||v||_{\HH1(\Omega)}||w||_{\HH1(\Omega)}.
\end{equation}
\end{proposition}
\begin{proof}
The proof of \eqref{ineq:a-ahell} is detailed in \cite[Proposition 6.3]{art-joyce-1}, given using \eqref{ineq:Ghr-Id_1/Jh-1} and~\eqref{ineq:AhJh}. To prove~\eqref{ineq:m-mhell}, consider $v, w \in \Vhlifte$. Using \eqref{eq:Jh-1=0=Dgh-Id} and \eqref{ineq:Ghr-Id_1/Jh-1}, we have, 
\begin{multline*}
     |(m-\mhell)(v,w)|  = |\int_{\Omega}  v w \ (1- \frac{1}{J^\ell_h}) \d x| 
     = |\int_{B_h^\ell} v w \ (1- \frac{1}{J^\ell_h}) \d x| \\
      \le \left\| 1- \frac{1}{J^\ell_h}\right\|_{\Linf(B_h^\ell)} \| v \|_{\L2(B_h^\ell)}\|  w \|_{\L2(B_h^\ell)} 
      \le c h^r \| v \|_{\L2(B_h^\ell)}\|  w \|_{\L2(B_h^\ell)}.
\end{multline*}
Since $v, w \in\Vhlifte \subset \HH1\omgam$, we apply \eqref{ineq:L2+h1/2} as follows,   
\begin{align*}
     |(m-\mhell)(v,w)| & \le c h^r  \left( h^{1/2} \|v\|_{\HH1(\Omega)} \right) \left( h^{1/2} \|w\|_{\HH1(\Omega)}\right) \le c h^{r+1} \|v\|_{\HH1(\Omega)} \|w\|_{\HH1(\Omega)}.
\end{align*}
\end{proof}
%
%
%
%
%
%
\begin{coro}
\label{coro:norm-equiv-with-h-r}
Considering a sufficiently small mesh size $h>0$, the boundary of the mesh domain~$\ghh$ is inside the tubular neighbourhood $\mathcal{U}_{\delta_\Gamma}$, defined in Proposition \ref{tub_neigh_orth_proj_prop}. Then,  there exists~$c>0$ such that,   
\begin{equation*}
\label{A_a}
    ||.||_{\ahell} \le (1+ch^r) ||.||_{a}, \ \ \ \ \ \  ||.||_{a} \le (1+ch^r)  ||.||_{\ahell}, 
\end{equation*}
\begin{equation*}
\label{M_m}
    ||.||_{\mhell} \le (1+ch^{r}) ||.||_{m}, \ \ \ \ \ \  ||.||_{m} \le (1+ch^{r})   ||.||_{\mhell},
\end{equation*}
where the norms $\|u\|_a, \|u\|_m, \|u\|_{\ahell}, \|u\|_{\mhell}$ are associated to the bilinear forms~$a$, $\ahell$, $m$ and $\mhell$, respectively. 
%
Consequently, the norms $||.||_{\ahell}$ and $||.||_{a}$ (resp.~$||.||_{\mhell}$ and $||.||_{m}$) are equivalent.
\end{coro}
\begin{proof} 
This proof is an adaptation of the proof of \cite[Corollary 2.3]{D4}, which is detailed for readers convenience. Let $u \in \HH1 \omgam$, one has,
$$
    ||u||_{a^\ell_h}^2- ||u||_{a}^2 = a^\ell_h(u,u)-a(u,u) = \big(a^\ell_h-a \big) (u,u).
$$
Then, we deduce that,
$$
    ||u||_{a^\ell_h}^2 
    \le  ||u||_{a}^2 + |\big( a_h^\ell-a \big) (u,u)| \le (1+ch^r) ||u||_{a}^2,
$$
where we used the geometric error estimation \eqref{ineq:a-ahell}. 
Taking its square rout, it follows that,
\begin{equation*}
     ||u||_{a_h^\ell} \le \sqrt{(1+ch^r)} ||u||_{a} \le (1+ch^r) ||u||_{a},
\end{equation*} 
since for any $x \ge 0 $, $ 1+x \le (1+\frac{1}{2}x)^2$.  
In a similar manner, the rest of the inequalities can be proved, by using \eqref{ineq:a-ahell} and \eqref{ineq:m-mhell}.
\end{proof}
\subsection{Preliminary eigenvalue estimate}
\label{sec:error-estimation-eigenvalue-1}
A preliminary eigenvalue error estimation is needed before proceeding with the error estimation. It has to be noted that a similar result was established in~\cite[Th. 3.3]{D4}, in a different context. For sake of completeness, we detail the proof of the following proposition. 
\begin{proposition}
%
%
%
Let $\lbdi$ be an exact eigenvalue of multiplicity $N$ of Problem~\eqref{fv_faible}, such that $\lbdj=\lbdi$, for any $j\in \J =\{i,...,i+N-1\}$. Let $\{ \lbdhp \}_{p=1}^{\dim(\Vh)}$ be the set of discrete eigenvalues relatively to Problem~\eqref{fvd}. Then, for any~$j \in \J$, there exists $\ci > 0$ such that,
\begin{equation}
\label{err:eigenvalue_initial}
    |\lbdj- \lbdhj  | \le \ci ( h^{2k} + h^{r}).
\end{equation}
\end{proposition}
\begin{proof}
Let $j \in \J$. To estimate the error $|\lbdj - \lbdhj|$, we introduce the following intermediate formulation: Find $(\Tilde{\lambda}, \Tilde{U})  \in \R^+\times \Vhlifte$, such that,
\begin{equation}
\label{fv-vhell-sur-a-m}
    a(\Tilde{U}, v)=\tilde{\lambda} m(\Tilde{U},v) \ \ \ \ \forall \ v \in \Vhlifte.
\end{equation} 
Problem~\eqref{fv-vhell-sur-a-m} has a finite number of solutions. Denote $\lbdtildep \in \R$ its eigenvalues, for $p=1,...,\dim(\mathbb{V}_h^\ell)$. Then for $j \in \J$, we separate the eigenvalue error as follows,
$$
   |\lbdj-\lbdhj|\le |\lbdj-\lbdtildej| +|\lbdtildej-\lbdhj|,
$$
and estimate each term separately. \medskip

For the estimation of $|\lbdj-\lbdtildej|$, note that the eigenvalue problem \eqref{fv-vhell-sur-a-m} is in a conformal, coercive and consistent setting. Indeed, the variationnal form is defined using the same bilinear forms $a$ and $m$ as in the initial formulation \eqref{fv_faible}, with the approximation space $\Vhlifte \subset \HH1 \omgam$. Thus, we refer to the detailed explanation in \cite[chapter 3.3]{EG} to obtain the following classical estimation, 
\begin{equation}
    \label{ineq:val-p-preuve-2}
    |\lbdj - \lbdtildej| \le  \ci h^{2k}.
\end{equation}
From the inequality \ref{ineq:val-p-preuve-2}, we notice that for any $j \in \J$, $0<\lbdtildej \le \ci$. 

To estimate $|\lbdtildej-\lbdhj|$, we proceed in a analogous way as in \cite[Lemma 3.1]{D4}. Note that, by \cite[Proposition 3.63]{EG}, the discrete eigenvalues can be written a follows,
\begin{equation}
\label{eq:def-lambda}
      \lbdhj= \min_{{E \in V_j}}\max _{v\in E 
      }R_{{a^\ell_h}}(v) 
      \quad \mbox{and} \quad
      \lbdtildej=\min_{{E \in V_j}}\max_{v\in E 
      } R_{a}(v),
\end{equation}
where the associated Rayleigh quotients are written as follows, 
\begin{equation*}
     R_{\ahell}(v) = \frac{{\ahell}(v,v)}{\mhell(v,v)} \quad \mbox{and} \quad R_a(v) = \frac{a(v,v)}{m(v,v)},
\end{equation*}
where $V_j$ is the set of all sub-spaces of $\Vhlifte$ of dimension $j$. 

Consider $E \in V_j$. By definition of the Rayleigh quotient and using the norm equivalence in Corollary \ref{coro:norm-equiv-with-h-r}, we can deduce for any $v \in E$,
\begin{equation*}
    R_{a_h^\ell}(v) = \frac{{a^\ell_h}(v,v)}{m_h^\ell(v,v)} \le \frac{(1+ch^r)^2 a(v,v)}{\frac{m(v,v)}{(1+ch^{r})^2}} = (1+ch^r)^4 R_a(v).
\end{equation*}
Using \eqref{eq:def-lambda}, it follows that, 
\begin{equation*}
    \lbdhj \le \min_{{E \in V_j}}\max_{v\in E}  (1+ch^r)^4 R_a(v) =  (1+ch^r)^4 \lbdtildej.
\end{equation*}
Then, $ \lbdhj  \le  \lbdtildej + ch^r \lbdtildej$, and we have, $\lbdhj  -  \lbdtildej \le ch^r \lbdtildej \le \ci h^r.$ In a similar manner, we can prove that $\lbdtildej - \lbdhj \le \ci h^r.$ To conclude, we combine these two inequalities as follows,
\begin{equation}
    \label{ineq:val-p-preuve}
     |\lbdhj - \lbdtildej|\le \ci h^r.
\end{equation}

To conclude, we combine \eqref{ineq:val-p-preuve} and \eqref{ineq:val-p-preuve-2} to arrive at \eqref{err:eigenvalue_initial}.
\end{proof}
\begin{remarque}
\label{exp:muJ}
    As a result of the estimation \ref{err:eigenvalue_initial}, the eigenvalues $\{\lbdj\}_{j\in \J}$ are only approximated by the set of discrete eigenvalues $\{ \lbdhj \}_{j\in \J}$. {Consequently, for a sufficiently small mesh step $h$, the following quantity, which appears in the eigenfunction estimations, is finite,}
\begin{equation*}
    \muJ = \max_{j \in \J} \max_{p \notin \J} |\frac{\lbdj}{\Lambda_p -\lbdj}|<\infty.
\end{equation*}
Additionally, the set of  eigenvalues $\{ \lbdhj  \}_{j \in \J}$ is separated from the rest of the continuous spectrum, i.e.,
\begin{equation*}
     \lambda_{i-1} < \Lambda_{i} \ \ \ \mbox{and}  \ \ \ \Lambda_{i+N-1} < \lambda_{i+N}.
\end{equation*}
Furthermore, this set of discrete eigenvalues $\{ \lbdhj  \}_{j \in \J}$ can be bounded independently from the mesh size $h$. Indeed, there exists $\ci>0$, such that, $|\lbdhj|\le \ci$, for all $j \in \J.$
We refer to \cite[page 6]{14}, \cite[\S 2.3]{23}, \cite[page 3]{eigenval1}, \cite[Section 3.2]{27} and \cite[remark 3.4]{D4}. 
\end{remarque}

\subsection{Eigenfunction error estimations}
\label{sub-sec:eigenvect-proof}
In this section, is presented the proof of the estimations \eqref{err:eigenvectors-L2} and \eqref{err:eigenvectors-H1} of Theorem \ref{th-error-bound-eigenfunction}.
To begin with, we recall that $\Fhlift = \oplus_{j \in \J} \Ehjlift=\oplus_{j \in \J}{\rm span}\{\uhjlifte\}$. We define the following projections, which are a useful tool in the eigenfunction error estimates (see \cite[\S 1.6.3]{EG}).
\begin{definition}
\label{def:projections}
We define the following projections:
\begin{itemize}
    \item Let $\projG : \HH1 \omgam \to \Vhlifte$ be the Riesz projection, such that $\forall \ v \in \HH1\omgam$, there exists a unique finite element function $\projG (v) \in \Vhlifte$ that satisfies, 
    \begin{equation*}
        \ahell(\projG (v),w) = \ahell(v,w), \quad \forall \ w \in \Vhlifte.
    \end{equation*}
%
%
%
%
    \item Let $\ProgZj : \HH1 \omgam \to \Fhlift $ be the orthogonal projection with respect to $\ahell$ onto $\Fhlift$, such that for all $v \in \HH1\omgam $, 
    \begin{equation*}
        \ahell(\ProgZj (v),w) = \ahell(v,w), \quad  \forall \ w \in \Fhlift.
    \end{equation*} 
     \item Let $\ProjHj: \HH1 \omgam  \to \Fhlift $ be the orthogonal projection with respect to $\mhell$ onto $\Fhlift$, such that for all $v \in \HH1 \omgam$, 
     \begin{equation*}
        \mhell(\ProjHj(v),w) = \mhell(v,w), \quad \forall \ w \in \Fhlift .
     \end{equation*}
\end{itemize}
\end{definition}
%
%
%
%
%
%
%
%
%
\begin{remarque}
\label{rem:express-Zlambdaj}
Note that the previous orthogonal projections satisfy the following relation (see \cite[Section 2.4]{D4} and \cite[Lemma 2.2]{23}),  $$\ProgZj =\ProjHj \circ \projG.$$
\end{remarque}
%
%
The key idea, in proof of the estimations \eqref{err:eigenvectors-L2} and \eqref{err:eigenvectors-L2}, is to separate the error in two terms for both norms as follows,
\begin{equation*}
    \inf_{U \in \Fhlift}\|\uj-U \ \|  \le \|\uj-\projG \uj \ \|+ \|\projG \uj - \ProgZj \uj\ \|.
\end{equation*}
The first term will be bounded using a classical interpolation result (see~\cite[\S 1.6.3]{EG}). If $\uj \in \Hk1 \omgam$, by definition of the Riesz projection $\projG$, there exists $c>0$ such that, 
\begin{equation}
\label{ineq:Gellu-u-ahell}
    \|\uj -\projG \uj  \|_{a_h^\ell} = \inf_{v \in \Vh^\ell} \|\uj  - v \|_{a_h^\ell} \le c h^k \| \uj \|_{\Hk1\omgam}.
\end{equation}
Using an Aubin-Nitsche argument as proved in Appendix \ref{appendix:proof-mhell-u-Gu}, there exists $c>0$ such that,
\begin{equation}
\label{ineq:Gellu-u-mhell}
  \|\uj -\projG \uj  \|_{\mhell}  \le c h^{k+1}.
\end{equation}
As for the second term, we recall that $\{ \uhplifte\}_{p=1}^{\dim (\Vh) }$ forms an orthonormal basis of $\Vhlifte$ with respect to $\mhell$. The lifted space finite element space can be decomposed as follows~$\Vhlifte:=\Fhlift \oplus \Shlift$, where~$\Fhlift := \oplus_{j \in \J} {\rm span}\uhjlifte$ and $\Shlift := \oplus_{p \not\in \J} {\rm span} \uhplifte$ are orthogonal spaces with respect to $\mhell$. We denote,
\begin{equation}
\label{expression:W}
    W := \projG \uj- \ProgZj \uj= \projG \uj- \ProjHj \circ\projG \uj.
\end{equation} Since $\ProjHj$ is the orthogonal projection over $\Fhlift$ with respect to $\mhell$, then we have,
\begin{equation*}
     W \in \Shlift, \quad \mhell(W, \uhjlifte) = 0, \quad \forall \, j \in \J.
\end{equation*}
Consequently, 
$W  =\sum_{p \notin \J} \beta_p \uhplifte$, where we denote the coefficients $\beta_p := \mhell(W, \uhplifte)$, for all $p \notin \J$. Then the $\mhell$ norm of $W$ is given as follows,
\begin{equation}
    \label{eq:W-beta}
        \|W\|^2_{\mhell} = \sum_{p\notin \J} \beta_p^2,
\end{equation}
since $\{ \uhplifte\}_{p\notin\J}$ forms an orthonormal basis of $\Shlift$ for the product $\mhell$.

\medskip

In the following propositions the $\ahell$ and $\mhell$ norms of $W$ will be evaluated in order to bound the error afterwards. We introduce the following notation, 
\begin{equation}
    \label{eq:Z}Z:=\sum_{p\not\in \J}\frac{\lbdi}{\Lambda_p - \lbdi}\beta_p \uhplifte,
\end{equation}
where $\Lambda_p$ is a discrete eigenvalue with its associated eigenfunction $\uhplifte$, and $\lbdi$ is the exact eigenvalue.
\begin{proposition}
\label{prop-norm-W-equat}
Let $j\in \Ji$ and $\uj$ be an exact eigenfunction associated with~$\lbdi$. 
The norms of~$W$, given in  \eqref{expression:W}, can be expressed as follows,
\begin{equation}
\label{eq:exp-W-mhell}
    \|W\|^2_{m_h^\ell}
    = m_h^\ell(\uj-\projG \uj,Z ) +(m- m_h^\ell)(\uj,Z ) 
    + \frac{1}{\lbdi} (a_h^\ell-a)(\uj,Z),
\end{equation}
\begin{equation}
\label{eq:exp-W-ahell}
     \|W\|^2_{a_h^\ell} = \lbdi m_h^\ell( \uj - \projG \uj, W) +\lbdi \|W\|^2_{\mhell} +\lbdi (m- m_h^\ell) ( \uj, W) + (a_h^\ell-a) (\uj, W).
\end{equation}
\end{proposition}
\begin{proof}
This proof is inspired from \cite[Lemma 4.1]{D4}, but for sake of completeness we detail it. The main difference is that in our case, we do not consider a surface problem as in \cite{D4} and the eigenfunctions $\{ \uj \}_{j \in \Ji}$ are on $\Omega$. 

By Equation \eqref{eq:W-beta}, the $\mhell$ norm of $W$ is written as follows,
\begin{align}
   \label{eq:mhell-Glifte-uj-Upell}
  \|W\|^2_{\mhell} & =\sum_{p \notin \J} \beta_p^2=\sum_{p \notin \J} \beta_p \mhell(W, \uhplifte).
\end{align}
To prove \eqref{eq:exp-W-mhell}, we try to estimate $\mhell(W, \uhplifte)$, 
for $p \notin \J$. 
By Remark~\ref{rem:express-Zlambdaj},~$\ProgZj = \ProjHj \circ \projG$, and we get for $p \notin \J$,
\begin{equation}
    \label{eq:Zlifte+Upell-mhell=0}
    m_h^\ell( \ProgZj  v , \uhplifte ) =  m_h^\ell(\ProjHj ( \projG v) , \uhplifte )=0, \qquad \forall \, v \in \HH1 \omgam.
\end{equation} 
%
%
%
%
%
%
%
%
%
Using \eqref{eq:Zlifte+Upell-mhell=0}, we get for $p \notin \J$,
\begin{equation*}
  \mhell(W , \uhplifte ) 
   =  \mhell(\projG  \uj -  \ProgZj  \uj , \uhplifte ) 
     =  m_h^\ell(\projG (\uj) , \uhplifte ).
\end{equation*} 
The next step is to estimate $\mhell \Big(\projG \uj,  \uhplifte  \Big)$. We denote $\Lambda_p$ a discrete eigenvalue associated to~$\uhplifte  \in \Shlift$ such that,
\begin{equation*}
   \Lambda_p  \mhell(V,\uhplifte) = \ahell(V,\uhplifte), \ \ \ \ \forall \ V \in \Vhlifte.
\end{equation*}
Taking in the latter equation $V=\projG  \uj \in \Vhlifte$, we get by using the definition of~$\projG$,
$$
   \Lambda_p  \mhell(\projG  \uj,\uhplifte) = \ahell(\projG  \uj,\uhplifte) 
    =\ahell(\uj,\uhplifte) 
    =a(\uj,\uhplifte) + (a_h^\ell-a)(\uj,\uhplifte).
$$
Since $\uj$ is an exact eigenfunction associated to $\lbdi$ of Problem \eqref{fv_faible}, we get, 
\begin{eqnarray*}
      \Lambda_p  \mhell(\projG  \uj,\uhplifte) & = &\lbdi m(\uj,\uhplifte) + (a_h^\ell-a)(\uj,\uhplifte) \\
   & =& \lbdi \mhell(\uj,\uhplifte) +\lbdi (m- \mhell)(\uj,\uhplifte)+ (a_h^\ell-a)(\uj,\uhplifte),
\end{eqnarray*}
where we added and subtracted $\lbdi \mhell(\uj,\uhplifte)$.

Subtracting $ \lbdi  \mhell(\projG  \uj,\uhplifte)$ on both sides of the equation, we get,
\begin{equation*}
    (\Lambda_p - \lbdi)  \mhell(\projG  \uj,\uhplifte) =\lbdi \mhell(\uj-\projG  \uj,\uhplifte) +\lbdi (m- \mhell)(\uj,\uhplifte)+ (a_h^\ell-a)(\uj,\uhplifte).
\end{equation*}
For any $p \notin \J$, $\Lambda_p - \lbdi \ne 0$, then we have,
\begin{multline*}
    \mhell(\projG  \uj, \uhplifte)  =  \frac{1}{\Lambda_p - \lbdi}\lbrace \lbdi \mhell(\uj-\projG  \uj,\uhplifte) +\lbdi (m- \mhell)(\uj,\uhplifte)  + (a_h^\ell-a)(\uj,\uhplifte) \rbrace \\
      =  \mhell(\uj-\projG  \uj,\frac{\lbdi}{\Lambda_p - \lbdi}\uhplifte) +(m- \mhell)(\uj,\frac{\lbdi}{\Lambda_p - \lbdi}\uhplifte) + \frac{1}{\lbdi} (a_h^\ell-a)(\uj,\frac{\lbdi}{\Lambda_p - \lbdi}\uhplifte).
\end{multline*}
To  arrive to \eqref{eq:exp-W-mhell}, we replace the latter expression in \eqref{eq:mhell-Glifte-uj-Upell} as follows, 
\begin{multline*}
        \|W\|^2_{\mhell}  = \sum_{p \in \{1,...dim(\Vh)\} \setminus \J} \beta_p \Big( \mhell(\uj-\projG \uj,\frac{\lbdi}{\Lambda_p - \lbdi}\uhplifte)\\
        +(m- \mhell)(\uj,\frac{\lbdi}{\Lambda_p - \lbdi}\uhplifte)
    + \frac{1}{\lbdi} (a_h^\ell-a)(\uj,\frac{\lbdi}{\Lambda_p - \lbdi}\uhplifte) \Big).
\end{multline*}

The proof of \eqref{eq:exp-W-ahell} is a tad similar to the latter one. Keeping in mind that~$W = (\I- \ProjHj )\projG  \uj$, its $\ahell$-norm is written as follows,
\begin{multline*}
    \|W\|_{a_h^\ell}^2  = a_h^\ell (W,W) 
     =a_h^\ell ((\I-\ProjHj )\projG  \uj, (\I-\ProjHj )\projG  \uj) \\
    = a_h^\ell (\projG  \uj, (\I-\ProjHj )\projG  \uj) - a_h^\ell (\ProjHj\projG  \uj, (\I-\ProjHj )\projG  \uj).
\end{multline*}
Note that, for any $V \in \Fhlift$, we have,
\begin{equation*}
    a_h^\ell ((\I-\ProjHj )\projG  \uj, V)   = a_h^\ell (\projG  \uj, V) - a_h^\ell (  \ProjHj \circ \projG \uj, V) 
     = a_h^\ell (\uj, V) - a_h^\ell ( \ProgZj   \uj, V) = 0,
\end{equation*}
where we used the definitions of the orthogonal projections $\projG$ and $\ProgZj$. 
Thus, taking~$V = \ProjHj\projG  \uj \in \Fhlift$, $a_h^\ell (\ProjHj\projG  \uj, (\I-\ProjHj )\projG  \uj) = 0$. Then, the latter equation becomes,
\begin{equation*}
    \|W\|_{a_h^\ell}^2
     =a_h^\ell (\projG  \uj, (\I-\ProjHj )\projG  \uj) 
     =a_h^\ell ( \uj, (\I-\ProjHj )\projG  \uj)
     = a_h^\ell ( \uj, W),
\end{equation*}
where we used the definition of the orthogonal projection $\projG$ with respect to $\ahell$, given in Definition~\ref{def:projections}.
Adding and subtracting $a(\uj, W)$, we get,
\begin{equation*}
    \|W\|_{a_h^\ell}^2=a_h^\ell ( \uj, W)  = a( \uj, W) + (a_h^\ell ( \uj, W)-a ( \uj, W))
      = \lbdi m( \uj, W) + (a_h^\ell-a) ( \uj, W)).
\end{equation*}
Since $\uj$ is an exact eigenfunction associated to $\lbdi$, the latter equation holds. Adding and subtracting $\lbdi \mhell( \uj, W)$, we have, 
\begin{equation}
    \label{eq:exp-W-ahell-proof}
\|W\|_{a_h^\ell}^2
= \lbdi \mhell( \uj, W) +\lbdi (m- m_h^\ell) ( \uj, W) + (a_h^\ell-a) ( \uj, W)).
\end{equation}
Notice that $W=\sum_{p \notin \J} \beta_p \uhplifte,$ then by applying \eqref{eq:Zlifte+Upell-mhell=0}, we have~$m_h^\ell(\ProgZj \uj, W)=0.$ Then, we notice that,
\begin{multline*}
    \lbdi m_h^\ell( \uj, W) = \lbdi m_h^\ell( \uj, W) - \lbdi m_h^\ell(\ProgZj \uj, W) \\
    = \lbdi m_h^\ell( \uj - \projG \uj, W) +\lbdi m_h^\ell( \projG \uj -  \ProgZj \uj, W) 
     = \lbdi m_h^\ell( \uj - \projG \uj, W) +\lbdi m_h^\ell( W, W),
\end{multline*}
where we added and subtracted $\lbdi m_h^\ell( \projG \uj, W)$.
Thus after replacing the latter equation in \eqref{eq:exp-W-ahell-proof}, we get exactly \eqref{eq:exp-W-ahell}, 
\begin{equation*}
    \|W\|_{a_h^\ell}^2 = \lbdi m_h^\ell( \uj - \projG \uj, W) +\lbdi \|W \|^2_{\mhell} +\lbdi (m- m_h^\ell) ( \uj, W) \\
    + (\ahell-a) ( \uj, W).
\end{equation*}
\end{proof}

The following proposition is one of the main novelties of this work. The general idea of bounding the norms of $W$ can be seen in \cite{D4,23,14} for different problems. 
\begin{proposition}
Under the assumptions of Proposition \ref{prop-norm-W-equat}, there exists $\ci~>~0$ such that,
\begin{equation}
\label{ineq:W-mhell_au-carre}
    \|W\|_{\mhell} \le \ci \|\uj - \projG \uj\|_{ \mhell}   + \ci h^{r+1/2}\|\uj\|_{{\H2\omgam}},
\end{equation}
\begin{equation}
\label{ineq:W-ahell_au-carre}
    \|W\|_{a_h^\ell} \le \ci \|\uj-\projG \uj \|_{\mhell}   + \ci  h^{r+1/2}\|\uj\|_{{\H2\omgam}},
\end{equation}
where the expression of $W$ is given in  \eqref{expression:W}.
\end{proposition}
\begin{proof}
This proof is decomposed into three steps. 
\begin{enumerate}
\item Using the geometric error estimates \eqref{ineq:a-ahell} and \eqref{ineq:m-mhell}, the $\mhell$-norm of $W$, given by \eqref{eq:exp-W-mhell}, can be estimated as follows,
\begin{eqnarray*}
      \|W\|^2_{\mhell} & = & \mhell(\uj - \projG \uj,Z ) +\bigg[(m-\mhell) + \frac{1}{\lbdi}(a_h^\ell-a) \bigg](\uj,Z )\\
    & \le & c  \|\uj - \projG \uj\|_{ \mhell}  \|Z \|_{\mhell} + c h^{r+1} \|\uj\|_{\HH1(\Omega)} \|Z \|_{\HH1(\Omega)} \\
    & & \quad + \frac{c}{\lbdi} \bigg(h^r \|\nabla \uj\|_{\L2(B_h^\ell)} \|\nabla Z  \|_{\L2(B_h^\ell)}+ h^{r+1}\|\uj\|_{{\HH1(\Gamma)}}\|Z  \|_{\HH1(\Gamma)}\bigg),
\end{eqnarray*}
where the expression of $Z$ is given in \eqref{eq:Z}.
Keeping in mind that the discrete eigenfunctions are $\mhell$-orthogonal, by Remark \eqref{exp:muJ} of $\mu_J$, we have,
\begin{equation}
\label{ineq:muJ-norme-W-mhell}
    \| Z   \|_{\mhell}^2  = \sum_{p\not\in \J} \big( \frac{\lbdi}{\Lambda_p - \lbdi} \big)^2 \beta_p^2 \| \uhplifte \|_{\mhell}^2  \le \muJ^2 \|W\|^2_{\mhell}.
\end{equation}

Since $\uhplifte $ is a discrete eigenfunction associated to $\Lambda_p$, then for any $q \ne p$, we have $a_h^\ell(\uhplifte ,U_q^\ell) = \Lambda_p \mhell (\uhplifte ,U_q^\ell).$
This implies that, that the discrete eigenfunctions $\{U_p^\ell\}_{p\notin \J}$ are $\ahell$-orthogonal, and that the following inequality holds,
$$
    \| Z   \|_{a_h^\ell}^2 \le \muJ^2 \|W\|^2_{a_h^\ell}.
$$
As a consequence, one can deduce the following,
\begin{equation}
    \label{ineq:muJ-norme-W-ahell}
\|\nabla Z  \|_{\L2(B_h^\ell)} \le \muJ\|W\|_{a_h^\ell}  \quad \mbox{ and } \quad \|Z  \|_{\HH1(\Gamma)} \le \muJ \|W\|_{a_h^\ell}.
\end{equation}
Additionally, we get,
\begin{align*}
    \|Z \|_{\HH1(\Omega)} & \le \|Z \|_{\ahell} + \|Z \|_{\mhell} \le \muJ \|W\|_{\ahell} + \muJ \|W\|_{\mhell}.
\end{align*}
Using the latter inequality alongside \eqref{ineq:muJ-norme-W-mhell} and \eqref{ineq:muJ-norme-W-ahell}, we get,
%
%
%
%
\begin{multline*}
      \|W\|^2_{\mhell} \le c \muJ \|\uj - \projG \uj\|_{ \mhell}  \|W\|_{\mhell} + ch^{r+1} \muJ \|\uj\|_{\HH1(\Omega)} (\|W\|_{\ahell} +  \|W\|_{\mhell}) \\
    + \frac{c}{\lbdi} \bigg(h^r \|\nabla \uj\|_{\L2(B_h^\ell)}  + h^{r+1}\|\uj\|_{{\HH1(\Gamma)}} \bigg) \muJ \|W\|_{a_h^\ell}.
\end{multline*}
Since the exact eigenfunctions $\uj$ belongs to $\H2\omgam$, by applying \eqref{ineq:H1+h1/2}, we obtain,
\begin{eqnarray*}
      \|W\|^2_{\mhell} & \le & c \muJ \|\uj - \projG \uj\|_{ \mhell}  \|W\|_{\mhell} + c h^{r+1} \muJ \|\uj\|_{\HH1(\Omega)} \|W\|_{\mhell} \\
     & & \quad + c(1+\frac{1}{\lbdi})\muJ h^{r+1} \| \uj\|_{\HH1\omgam} \|W\|_{a_h^\ell} + c\muJ\frac{1}{\lbdi}h^{r+1/2} \| \uj\|_{\H2\omgam} \|W\|_{a_h^\ell}\\
     & \le & c \muJ \|\uj - \projG \uj\|_{ \mhell}  \|W\|_{\mhell} + c h^{r+1} \muJ \|\uj\|_{\HH1(\Omega)} \|W\|_{\mhell} \\
     & & \quad + c\muJ h^{r+1} \| \uj\|_{\HH1\omgam} \|W\|_{a_h^\ell} + c\muJ\frac{1}{\lbdi}h^{r+1/2} \| \uj\|_{\H2\omgam} \|W\|_{a_h^\ell}.
\end{eqnarray*}
%
%
%
%
%
%
%
%
%
%
%
%
Young's inequality, which states that, for all $\epsilon>0$, $ab \le \frac{a^2}{\epsilon^2} + \epsilon^2 b ^2$, is applied in the following inequality multiple times as follows, for $\epsilon_1,~\epsilon_2~>~0~,$ 
\begin{eqnarray*}
      \|W\|^2_{\mhell} & \le & 4c\muJ^2 \|\uj - \projG \uj\|^2_{ \mhell}  +   {\frac{1}{4}\|W\|^2_{\mhell} } + 4c h^{2r+2} \muJ^2 \|\uj\|^2_{\HH1(\Omega)}+   {\frac{1}{4} \|W\|_{\mhell}^2} \\
    & & \quad +  {\frac{c}{\epsilon_1^2} \muJ^2 h^{2r+2}\|\uj\|^2_{{\HH1\omgam}}} + \epsilon_1^2   \|W\|^2_{a_h^\ell} +  {\frac{c}{\epsilon_2^2} \muJ^2 h^{2r+1}\|\uj\|^2_{{\H2\omgam}}} +\epsilon_2^2 \frac{1}{\lbdi^2}   \|W\|^2_{a_h^\ell}\\
    & \le & c\muJ^2 \|\uj - \projG \uj\|^2_{ \mhell}  +   {\frac{1}{2}\|W\|^2_{\mhell}} +  {c\big(\frac{1}{\epsilon_1^2}+\frac{1}{\epsilon_2^2} \big) h^{2r+1} \muJ^2 \|\uj\|^2_{\H2\omgam}} \\
    & & \quad + (\epsilon_1^2  + \frac{\epsilon_2^2}{\lbdi^2})   \|W\|^2_{a_h^\ell}.
\end{eqnarray*}
Thus, we arrive at,
\begin{align}
    \label{ineq:mhell-norm-W}
    \|W\|^2_{\mhell} & \le c\muJ^2 \|\uj - \projG \uj\|^2_{ \mhell}   + c \muJ^2 h^{2r+1}\|\uj\|^2_{{\H2\omgam}} +   {(\epsilon_1^2  + \frac{\epsilon_2^2}{\lbdi^2}) \|W\|^2_{\ahell}}.
\end{align}
It remains to bound $  {\|W\|^2_{\ahell}}$.
%
%
%
%
%
%
%
%
%
%
\item To estimate the $\ahell$-norm of $W$, we first recall \eqref{eq:exp-W-ahell} and we use the geometric error estimates~\eqref{ineq:a-ahell} and \eqref{ineq:m-mhell} as follows,
\begin{multline*}
    \|W\|_{a_h^\ell}^2 \le c\lbdi \|\uj-\projG \uj \|_{\mhell} \| W\|_{\mhell} + \lbdi \| W\|_{\mhell}^2 + c \lbdi h^{r+1} \|\uj\|_{\HH1(\Omega)}\|W\|_{\HH1(\Omega)} \\
    + c h^r \|\nabla \uj\|_{\L2(B_h^\ell)} \|\nabla W \|_{\L2(B_h^\ell)} + ch^{r+1} \| \uj\|_{\HH1(\Gamma)} \| W \|_{\HH1(\Gamma)}.
\end{multline*}
Since $\uj$ belongs to $\H2\omgam$, the inequality \eqref{ineq:H1+h1/2} is applied as follows,
\begin{eqnarray*}
     \|W\|_{a_h^\ell}^2 & \le & c \lbdi \|\uj-\projG \uj \|_{\mhell} \| W\|_{\mhell} +  \lbdi \| W\|_{\mhell}^2 + c \lbdi h^{r+1} \|\uj\|_{\HH1(\Omega)}(\|W\|_{\mhell }+ {\|W\|_{\ahell }})\\
   & & \quad +  {c h^{r+1/2} \| \uj\|_{\H2\omgam} \| W \|_{a_h^\ell}+ch^{r+1}\| \uj\|_{\HH1(\Gamma)} \| W \|_{a_h^\ell}} \\
    &\le &c \lbdi \|\uj-\projG \uj \|_{\mhell} \| W\|_{\mhell} +  \lbdi \| W\|_{\mhell}^2 + c \lbdi h^{r+1} \|\uj\|_{\HH1(\Omega)}\|W\|_{\mhell }\\
    & & \quad + { c(1+\lbdi) h^{r+1/2}  \| \uj\|_{\H2\omgam} \| W \|_{a_h^\ell}}.
\end{eqnarray*}
Young's inequality is applied as follows,
%
%
%
%
\begin{multline*}
    \|W\|_{a_h^\ell}^2 \le 4c \lbdi \|\uj-\projG \uj \|^2_{\mhell}  {+ \frac{1}{4} \lbdi \| W\|^2_{\mhell} +  \lbdi \| W\|_{\mhell}^2} +  { 4c \lbdi^2 h^{2r+2} \|\uj\|^2_{\HH1(\Omega)}}\\ 
    + {\frac{1}{4}\|W\|^2_{\mhell}} 
    +   {4c (1+\lbdi)^2 h^{2r+1}  \| \uj\|^2_{\H2(\Omega) }} + \frac{1}{4}\|W\|^2_{a_h^\ell}.
\end{multline*}
Then, we deduce,
\begin{equation}
\label{ineq:ahell-norm-W}
    \|W\|_{a_h^\ell}^2 \le c \lbdi \|\uj-\projG \uj \|^2_{\mhell}  {+ c (1+\lbdi) \| W\|^2_{\mhell}}   + 
      {c h^{2r+1} (\lbdi^2 + (1+\lbdi)^2)  \| \uj\|^2_{\H2\omgam }}.
\end{equation}
Using the estimation \eqref{ineq:ahell-norm-W} in the inequality \eqref{ineq:mhell-norm-W}, we get,
\begin{eqnarray*}
    \|W\|^2_{\mhell} & \le & c\muJ^2 \|\uj - \projG \uj\|^2_{ \mhell}   + c \muJ^2 h^{2r+1}\|\uj\|^2_{{\H2\omgam}} +   {(\epsilon_1^2  + \frac{\epsilon_2^2}{\lbdi^2}) \|W\|^2_{\ahell}}\\
    & \le & c\muJ^2 \|\uj - \projG \uj\|^2_{ \mhell}   + c \muJ^2 h^{2r+1}\|\uj\|^2_{{\H2\omgam}} +   { c(\epsilon_1^2  + \frac{\epsilon_2^2}{\lbdi^2}) \lbdi \|\uj-\projG \uj \|^2_{\mhell}} \\
    & & \quad  { + c (\epsilon_1^2  + \frac{\epsilon_2^2}{\lbdi^2})(1+\lbdi) \| W\|^2_{\mhell}  + c h^{2r+1}(\epsilon_1^2  + \frac{\epsilon_2^2}{\lbdi^2})(\lbdi^2 + (1+\lbdi)^2)  \| \uj\|^2_{\H2\omgam }}\\ 
    & \le & {c\muJ^2 \|\uj - \projG \uj\|^2_{ \mhell}}   +   {c \muJ^2 h^{2r+1}\|\uj\|^2_{{\H2\omgam}}} +  { c(\lbdi  + \frac{1}{\lbdi}) \|\uj-\projG \uj \|^2_{\mhell}} \\
    & & \quad + c (\epsilon_1^2  + \frac{\epsilon_2^2}{\lbdi^2})(1+\lbdi) \| W\|^2_{\mhell} +  { c h^{2r+1}(1  + \frac{1}{\lbdi^2}) (\lbdi^2 + (1+\lbdi)^2)  \| \uj\|^2_{\H2\omgam }}.
\end{eqnarray*}
Taking $\epsilon_1 = \frac{1}{2}\sqrt{(\frac{1}{c(1+\lbdi)})}$ and $\epsilon_2 = \frac{\lbdi}{2}\sqrt{(\frac{1}{c(1+\lbdi)})}$, these quantities will satisfy the following inequality, $$1- (\epsilon_1^2  + \frac{\epsilon_2^2}{\lbdi^2})(1+\lbdi)>0.$$ 
Then we have,
\begin{align*}
    \|W\|^2_{\mhell} & \le  {\ci\|\uj - \projG \uj\|^2_{ \mhell}}   +   {c'_{\lbdi} h^{2r+1}\|\uj\|^2_{{\H2\omgam}}},
\end{align*}
where $\ci = c(\muJ^2+\lbdi+\frac{1}{\lbdi})$ and $c'_{\lbdi} = c(\muJ^2+(1  + \frac{1}{\lbdi^2}) (\lbdi^2 + (1+\lbdi)^2))$. To arrive to the inequality \eqref{ineq:W-mhell_au-carre}, we take the square root of the latter inequality.

%
%
%
%
%
%
%
%
%
%
%
\item Lastly we also need to estimate the $\ahell$ norm of $W$. We use the estimations~\eqref{ineq:ahell-norm-W} and \eqref{ineq:W-mhell_au-carre} as follows,
\begin{eqnarray*}
    \|W\|_{a_h^\ell}^2 
    & \le & c \lbdi \|\uj-\projG \uj \|^2_{\mhell}+ c (1+\lbdi) \| W\|^2_{\mhell}  + c h^{2r+1} (\lbdi^2 + (1+\lbdi)^2)  \| \uj\|^2_{\H2\omgam } \\ 
    & \le & c\lbdi \|\uj-\projG \uj \|^2_{\mhell} 
     + c(\lbdi^2 + (1+\lbdi)^2) h^{2r+1}  \| \uj\|^2_{\H2(\Omega) } \\
    & & \quad + c(\lbdi + 1) \bigg(\ci\|\uj - \projG \uj\|^2_{ \mhell}   + c'_{\lbdi} h^{2r+1}\|\uj\|^2_{{\H2\omgam}} \bigg).
\end{eqnarray*}
Consequently we arrive at the desired result by taking its square root,
\begin{equation*}
     \|W\|_{a_h^\ell}
     \le C_{\lbdi}\|\uj - \projG \uj\|_{ \mhell} + C'_{\lbdi} h^{r+1/2}\|\uj\|_{{\H2\omgam}},
\end{equation*}
where $C_{\lbdi}=\sqrt{ c \big( \lbdi + (\lbdi + 1)\ci\big)}$ and $C'_{\lbdi}=\sqrt{c \big(\lbdi^2 + (1+\lbdi)^2 + (\lbdi + 1)c'_{\lbdi}\big)}$.
\end{enumerate}
\end{proof}
%
%
%
%
%
%
%
%
%
%
%
\begin{remarque}
In this work, the function $W= \projG \uj- \ProgZj \uj$ being a linear combination of lifted discrete eigenfunctions is in the lifted finite element space $\Vhlifte$, which is a subspace of $\HH1\omgam$, therefore $W$ is not necessarily in $\H2 (\Omega)$. However if, by considering other finite element method like Hermite, $W$ will be in $\H2 \omgam$, then the inequality~\eqref{ineq:W-mhell_au-carre} may be improved as follows,
    \begin{equation*}
        \|W\|^2_{\mhell} \le \ci \|\uj - \projG \uj\|^2_{ \mhell}   + \ci h^{2r+2}\|\uj\|^2_{{\H2\omgam}}.
    \end{equation*}
    This may lead to a higher geometric error rate in the final error estimation for the~$\L2$ norm. However, notice that this conjecture should be checked carefully (but this is not the topic of the present paper).
\end{remarque}

The last step would be to combine all the previous results to estimate the eigenfunctions. 
%
%
%
%
\paragraph{Proof of Theorem \ref{th-error-bound-eigenfunction}: the estimates \eqref{err:eigenvectors-L2} and \eqref{err:eigenvectors-H1}.}
%
%
To prove \eqref{err:eigenvectors-H1}, we start by adding and subtracting $\projG \uj $ as follows,
$$
    \|\uj - \ProgZj \uj \|_{a_h^\ell}  \le \|\uj -\projG \uj \|_{a_h^\ell} + \|  \projG \uj- \ProgZj \uj\|_{a_h^\ell}  = \|\uj -\projG \uj \|_{a_h^\ell} + \|W\|_{a_h^\ell}.
$$
The latter inequality is obtained by definition of $W =  \projG \uj- \ProgZj \uj$. Applying respectively \eqref{ineq:W-ahell_au-carre}, \eqref{ineq:Gellu-u-mhell} and \eqref{ineq:Gellu-u-ahell}, we get,
\begin{align*}
     \|\uj - \ProgZj \uj \|_{a_h^\ell} & \le c\|\uj -\projG \uj \|_{a_h^\ell}  + C_{\lbdi} \|\uj -\projG \uj  \|_{\mhell}+C'_{\lbdi} h^{r+1/2}\|\uj \|_{{\H2\omgam}}  \\
     & \le \ci ( h^k + h^{r+1/2}).
\end{align*}
By the norm equivalence between $\| \cdot \|_{\ahell}$ and $\|\cdot \|_{\HH1 \omgam}$, the latter inequality leads to \eqref{err:eigenvectors-H1}.

Since $\ProjHj$ is the orthogonal projection with respect to $\mhell$ onto $\Fhlift$, then $\ProjHj \uj $ is the closest point to~$\uj $ with respect to the $\mhell$-norm. Since $\ProgZj =\ProjHj \circ \projG$ as mentioned in Remark \ref{rem:express-Zlambdaj}, we have, 
$$
    \|\uj  - \ProjHj \uj  \|_{\mhell} \le \|\uj  -  \ProgZj \uj  \|_{\mhell}  \le \|\uj  - \projG \uj  \|_{\mhell} +  \|\projG \uj  -  \ProgZj \uj  \|_{\mhell}.
$$
We apply \eqref{ineq:W-mhell_au-carre} and \eqref{ineq:Gellu-u-mhell} respectively to conclude,
$$
    \|\uj  - \ProjHj \uj  \|_{\mhell} \le c_{\lbdi} \|\uj  - \projG \uj \|_{ \mhell}   + c'_{\lbdi} h^{r+1/2}\|\uj \|_{{\H2\omgam}} \le \ci (  h^{k+1} + h^{r+1/2}).
$$
By the norm equivalence between $\| \cdot \|_{\mhell}$ and $\|\cdot \|_{\L2(\Omega)}$, the latter inequality leads to \eqref{err:eigenvectors-L2}.
%
%
%
%
%
%
%
%
%
%
%
%
%

\subsection{Eigenvalue error estimate}
 \label{sec:error-estimation-eigenvalue-2}
We recall that $\lbdi$ is an exact eigenvalue of multiplicity $N$ of Problem~\eqref{fv_faible}, such that $\lbdj=\lbdi$, for any $j\in \J=\{i,...,i+N-1\}$. In order to improve the preliminary eigenvalue error estimation~\eqref{err:eigenvalue_initial}, we introduce $ \projPj  : \Vhlifte \to \F $ the orthogonal projection with respect to $m$ onto the space $\F$,  such that for all~$v\in \Vhlifte$,
\begin{equation*}
  m(\projPj v ,t) = m(v,t), \ \ \ \forall \ t \in \F .
\end{equation*}

The idea of the following lemma can be found in \cite[Lemma 2.3]{Babuska-osborn-87} and \cite[Lemma 3.1]{Babuska-osborn-89}. However the main difference here is that we need to take into consideration the geometric error (see \cite[Lemma 6.1]{D4} and \cite[Lemma 5.1]{5-th5.1}).
\begin{lem}[eigenvalue bound] 
\label{lem:eigenval_bound}
Let $U_j^\ell$ be a discrete eigenfuntion in $\Fhlift$ associated to $\lbdhj$ such that $\| U_j^\ell \|_{m}=~1$. Thus, the following inequality holds,
\begin{equation}
\label{lamda-lambdah}
    |\lbdj-\lbdhj| \le \|  \projPj U_j^\ell -U_j^\ell \|_a^2 + \lbdj \| \projPj U_j^\ell -U_j^\ell\|_m^2 +|a_h^\ell-a|(U_j^\ell,U_j^\ell) +\lbdhj|\mhell-m|(U_j^\ell,U_j^\ell).
\end{equation}
\end{lem}
\begin{proof}
First of all, we need to notice that $ \projPj U_j^\ell $ is in $\F$, 
thus,
\begin{equation}
\label{eq:a=lambdai.m}
    a( \projPj U_j^\ell , v ) = \lbdj m( \projPj U_j^\ell,v), \  \ \ \forall \ v \in \HH1\omgam.
\end{equation}
Taking $v = U_j^\ell \in \HH1\omgam$ in  \eqref{eq:a=lambdai.m}, we have,
\begin{equation*}
\label{star1}
    a( \projPj U_j^\ell , U_j^\ell ) = \lbdj m( \projPj U_j^\ell, U_j^\ell ).
\end{equation*}
Afterwards, taking $v = \projPj  U_j^\ell \in \HH1\omgam$  in  \eqref{eq:a=lambdai.m}, we get,
\begin{equation*}
\label{star2}
    \| \projPj U_j^\ell \|^2_{a} = \lbdj \| \projPj U_j^\ell\|_m^2.
\end{equation*}
Applying the latter two equations in the following estimation, we get,
\begin{multline*}
     \| \projPj U_j^\ell-U_j^\ell \|^2_{a} - \lbdj \| \projPj U_j^\ell - U_j^\ell\|_m^2 \\
\begin{array}{rcl}
    & = & \| \projPj U_j^\ell \|^2_{a} + \|U_j^\ell \|^2_{a} -2a(U_j^\ell, \projPj U_j^\ell) - \lbdj \|U_j^\ell\|_m^2 - \lbdj \| \projPj U_j^\ell \|_m^2 +2\lbdj m(U_j^\ell, \projPj U_j^\ell)\\
    & = &  \|U_j^\ell \|^2_{a}  - \lbdj \|U_j^\ell\|_m^2.
\end{array}
\end{multline*}
Since $\|U_j^\ell\|_m=1$, we have, 
\begin{equation*}
    - \lbdj =   \| \projPj U_j^\ell-U_j^\ell \|^2_{a} - \lbdj \| \projPj U_j^\ell - U_j^\ell\|_m^2 - \|U_j^\ell \|^2_{a}.
\end{equation*}
Keeping in mind that $\ahell(U_j^\ell, U_j^\ell) = \lbdhj \ \mhell(U_j^\ell, U_j^\ell) $, we get by adding and subtracting $\ahell(\uhjlifte,\uhjlifte)$,
\begin{eqnarray*}
    - \lbdj & = &  \| \projPj U_j^\ell-U_j^\ell \|^2_{a} - \lbdj \| \projPj U_j^\ell - U_j^\ell\|_m^2 - a(U_j^\ell, U_j^\ell) + \ahell(U_j^\ell, U_j^\ell) -  \lbdhj \ \mhell(U_j^\ell, U_j^\ell)\\
    &  = & \| \projPj U_j^\ell-U_j^\ell \|^2_{a} - \lbdj \| \projPj U_j^\ell - U_j^\ell\|_m^2 + (\ahell-a)(U_j^\ell, U_j^\ell)  -  \lbdhj \ \mhell(U_j^\ell, U_j^\ell).
\end{eqnarray*}
Since $m(U_j^\ell, U_j^\ell)=1$, then by adding $\lbdhj m(U_j^\ell, U_j^\ell)$ to each side of this equation, we have,
\begin{equation*}
   \lbdhj - \lbdj =   \| \projPj U_j^\ell-U_j^\ell \|^2_{a} - \lbdj \| \projPj U_j^\ell - U_j^\ell\|_m^2 + (\ahell-a)(U_j^\ell, U_j^\ell) +  \lbdhj (m-\mhell) (U_j^\ell, U_j^\ell).
\end{equation*}
By taking the absolute value of the latter equation and bounding it, we get~\eqref{lamda-lambdah}.
%
\end{proof}
The proofs of the following lemma and corollary are analogous to the proofs of \cite[Lemma~4.4 - Proposition~4.5]{D4}, which were given for a surface problem. For readers convenience, we will detail these proofs, and we recall that there exists $\ci>0$, such that $0<\Lambda_{j}  \le  \ci$, for all $j\in \J$. 
\begin{lem}
\label{lem:bound_Pmhell}
Following the assumption in Lemma \ref{lem:eigenval_bound}, there exists $\ci>0$ such that,
\begin{equation}
\label{ineq:H}
    \|\ProjHj v \|_{a_h^\ell} \le {\ci} \|v\|_{\mhell}, \quad \forall \, v \in \HH1 \omgam,
\end{equation}
where $\ProjHj$ is the orthogonal projection with respect to $\mhell$ onto $\Fhlift$, given in Definition \ref{def:projections}.
\end{lem}
\begin{proof}
Notice that $\ProjHj v\in \Fhlift = \underset{j \in \J}{\oplus} \Ehjlift $, then there exists constants $\beta_j \in \R$ for $j\in \J$ such that~$\ProjHj v = \sum_{j \in \Ji} \beta_jU_j^\ell.$ One can estimate its norm as follows,
\begin{multline*}
    \|\ProjHj v  \|_{a_h^\ell}^2  = a_h^\ell(\ProjHj v ,\ProjHj v )
     = a_h^\ell(\sum_{j \in \Ji} \beta_jU_j^\ell ,\ProjHj v )  \\
     = \sum_{j \in \Ji} \beta_j a_h^\ell(U_j^\ell ,\ProjHj v ) 
      = \sum_{j \in \Ji} \beta_j \lbdhj \mhell(U_j^\ell ,\ProjHj v).
\end{multline*}
Since $0<\lbdhj \le \ci$, for all $j \in \Ji$, we have, 
\begin{align*} 
    \|\ProjHj v \|_{a_h^\ell}^2 & \le \ci \sum_{j \in \Ji} \beta_j\mhell(U_j^\ell ,\ProjHj v ) 
     = \ci  \mhell( \ProjHj v, \ProjHj v ) = \ci   \|\ProjHj v \|_{\mhell}^2.
\end{align*}
Finally, by definition of the orthogonal projection $\ProjHj$, we conclude the proof as follows,
$$
  \|\ProjHj v  \|_{a_h^\ell}^2 \le \ci \|v\|_{\mhell}^2.
$$
\end{proof}
%
%
%
%
\begin{coro}
Following the assumptions of lemma \ref{lem:bound_Pmhell}, this inequality holds for any exact eigenfunction $\uj$ associated to $\lbdi$, there exists $\ci>0$ such that,
%
%
%
%
\begin{equation}
    \label{ineq:norm(u-Hu)-Z-G}
    \|\uj-\ProjHj \uj \|_{a_h^\ell} \le  \|\uj-\ProgZj \uj\|_{a_h^\ell} +\ci \|\uj-\projG \uj \|_{\mhell},
\end{equation}
where $\projG$ is the orthogonal projection with respect to $\ahell$ onto $\Vhlifte$ and $\ProgZj$ is the orthogonal projection with respect to $\ahell$ onto $\Fhlift$, given in Definition \ref{def:projections}.
\end{coro}
\begin{proof}
By adding and subtracting $\ProgZj \uj$, we have,
\begin{equation*}
    \|\uj-\ProjHj \uj\|_{a_h^\ell} \le \|\uj-\ProgZj \uj\|_{a_h^\ell} + \|\ProgZj \uj -\ProjHj \uj \|_{a_h^\ell}.
\end{equation*}
Since $\ProgZj  = \ProjHj \circ\projG$, we get, 
\begin{align*}
    \|\uj-\ProjHj \uj \|_{a_h^\ell} & \le \|\uj-\ProgZj \uj \|_{a_h^\ell} + \|\ProjHj \circ \projG \uj-\ProjHj \uj\|_{a_h^\ell}\\
    & = \|\uj-\ProgZj \uj \|_{a_h^\ell} + \|\ProjHj (\projG \uj -\uj)\|_{a_h^\ell}.
\end{align*}
To sum up, we apply \eqref{ineq:H} to arrive at \eqref{ineq:norm(u-Hu)-Z-G}.
\end{proof}
The error between a discrete eigenfunction and its projection onto the space spanned by the exact eigenfunctions is estimated in the following lemmas using~$\projPj$ the orthogonal projection with respect to $m$ onto the space $\F$. 

By \cite[Lemma 5.1]{23}, for a sufficiently small $h$, $\{ \ProjHj \up, p \in \J \}$ forms a basis for~$\Fhlift $. Since $\uhjlifte \in \Fhlift= {\rm span}\{ \ProjHj \up, p \in \J  \}$, it can be written as follows,
\begin{equation}
\label{eq:Uell}
    \uhjlifte  = \sum_{p \in \J} \alpha_p\ProjHj \up.
\end{equation}
Indeed, this can be traced back to the lower semicontinuity of the rank application and the fact that $\ProjHj \up$ tends to $\up$ as $h$ tends to $0$, for all $p\in \J$.
\begin{lem}
Let $\uhj$ be a  discrete eigenfunction associated to $\lbdhj$, such that~$\|\uhjlifte\|_m =1$. Then, we have,
\begin{equation}
\label{eq:Plbdu-u}
    \projPj \uhjlifte  - \uhjlifte  = \sum_{p\in \J} \alpha_p \Bigg[\sum_{t\in \J}  m(\ProjHj \up -\up ,u_t)u_t  +  (\up-\ProjHj \up) \Bigg],
\end{equation}
where $\{ \up \}_{p \in \Ji}$ denotes an orthonormal basis of $\F$ with respect to $m$ (thus made of exact eigenfunctions associated to~$\lbdi$).
\end{lem}
\begin{proof}
We will proceed as in \cite[lem 6.2 - 6.3]{D4}. 
We need to keep in mind that $\projPj$ is the orthogonal projection with respect to $m$ on $\F$. This implies that $\projPj \uhjlifte$ can be written as follows,
\begin{equation}
\label{exp:P_m-Ujell}
    \projPj \uhjlifte   = \sum_{t\in \J} m(\uhjlifte  ,u_t)u_t \in \F .
\end{equation}
Subtracting \eqref{eq:Uell} from the latter equation \eqref{exp:P_m-Ujell}, we get,
\begin{equation}
\label{U-PU}
    \projPj \uhjlifte   - \uhjlifte   = \sum_{t\in \J} m( \sum_{p\in \J} \alpha_p\ProjHj \up ,u_t)u_t - \sum_{p\in \J} \alpha_p\ProjHj \up.
\end{equation}
Since $m(\up,u_t)=\delta_{pt}$ for all $p, t \in \J$, we have,   
\begin{equation*}
  -  \sum_{p\in \J} \alpha_p m(\up,\up)\up + \sum_{p\in \J} \alpha_p\up =0.
\end{equation*}
Inserting this in \eqref{U-PU}, we get,
\begin{eqnarray*}
    \projPj \uhjlifte   - \uhjlifte   & = & \sum_{t\in \J} m( \sum_{p\in \J} \alpha_p\ProjHj \up ,u_t)u_t -  \sum_{p\in \J} \alpha_p m(\up,\up)\up +  \sum_{p\in \J} \alpha_p(\up-\ProjHj \up ) \\
    & = & \sum_{t\in \J} \sum_{p\in \J} \alpha_p m(\ProjHj \up  -\up ,u_t)u_t  +  \sum_{p\in \J} \alpha_p(\up-\ProjHj \up ) \\
    & = & \sum_{p\in \J}\alpha_p \Bigg[\sum_{t\in \J}  m(\ProjHj \up  -\up ,u_t)u_t  +  (\up-\ProjHj \up ) \Bigg].
\end{eqnarray*}
\end{proof}
\begin{lem}
\label{lem-equiv-eigenvector-estimations}
Let $\uhj$ be an eigenfunction associated to $\lbdhj$ such that~$\|\uhjlifte\|_m=1$. Then, for a sufficiently small mesh size $h$, there exists $\ci>0$ such that,
\begin{equation}
\label{ineq:norme-a-max}
    \|\uhjlifte  -\projPj \uhjlifte  \|_a \le \ci \max_{p\in \J} \|\up-\ProjHj \up \|_a,
\end{equation}
\begin{equation}
\label{ineq:norme-m-max}
    \|\uhjlifte  -\projPj \uhjlifte  \|_m \le \ci \max_{p\in \J} \|\up-\ProjHj \up  \|_m,
\end{equation}
\begin{equation}
\label{ineq:U-PU-norme-a}
    \|\uhjlifte  -\projPj \uhjlifte  \|_a \le  \ci (h^k + h^{r+1/2}),
\end{equation}
\begin{equation}
\label{ineq:U-PU-norme-m}
    \|\uhjlifte  -\projPj \uhjlifte  \|_m \le \ci (h^{k+1} + h^{r+1/2}).
\end{equation}
where $\ProjHj$ is the orthogonal projection over $\Fhlift$ with respect to $\mhell$, given in Definition \ref{def:projections}.
\end{lem}
\begin{proof}
Taking the norm with respect to the bilinear form a of \eqref{eq:Plbdu-u}, we bound it as follows,
\begin{equation*}
    \| \projPj \uhjlifte  - \uhjlifte   \|_a \le \sum_{p\in \J} |\alpha_p |  \Bigg[ \sum_{t\in \J}  |m( \ProjHj \up  -\up ,u_t)| \|u_t\|_a  +  \|\up-\ProjHj \up \|_a \Bigg].
\end{equation*}
By applying Cauchy-Schwarz, we have,
\begin{equation*}
    \| \projPj \uhjlifte  - \uhjlifte   \|_a \le (\sum_{p\in \J} |\alpha_p|^2 )^{\frac{1}{2}} \Bigg(\sum_{p\in \J}  \Bigg[ \sum_{t\in \J}  |m(\ProjHj \up  -\up ,u_t)| \|u_t\|_a  +  \|\up-\ProjHj \up \|_a \Bigg]^2 \Bigg)^{\frac{1}{2}}.
\end{equation*}
By Lemma 5.1 of \cite{23} the coefficients $(\alpha_p)_{p\in \J}$ satisfy, $\sum_{p\in \J} |\alpha_p|^2 \le C(N),$ where $C(N)$ is a constant dependent on the multiplicity $N$ of $\lbdi$. Keeping in mind that, for all $t \in  \J$, $u_t$ satisfies that, $a(u_t,v)=\lbdi m(u_t,v),$ for any~$v \in \HH1 \omgam,$ we have,
\begin{eqnarray*}
    \| \projPj \uhjlifte  - \uhjlifte   \|_a & \le & (C(N))^{\frac{1}{2}} \Bigg(\sum_{p\in \J}  \Bigg[ \sum_{t\in \J}  \frac{1}{\lbdi}|a( \ProjHj \up  -\up ,u_t)| \|u_t\|_a  +  \|\up-\ProjHj \up\|_a \Bigg]^2 \Bigg)^{\frac{1}{2}} \\
    & \le &\ci \Bigg(\sum_{p\in \J}  \Bigg[ \sum_{t\in \J}  \frac{1}{\lbdi}\|\ProjHj \up  -\up\|_a \|u_t\|_a \|u_t\|_a  +  \|\up-\ProjHj \up\|_a \Bigg]^2 \Bigg)^{\frac{1}{2}} \\
    & \le &\ci \Bigg(\sum_{p\in \J}  \Bigg[ \sum_{t\in \J}  \|\ProjHj \up  -\up\|_a \frac{1}{\lbdi}\|u_t\|_a^2 + \|\up-\ProjHj \up\|_a \Bigg]^2 \Bigg)^{\frac{1}{2}}.
\end{eqnarray*}
Noticing that $\frac{1}{\lbdi}\|u_t\|_a^2= \|u_t\|_m^2 = 1$ for all~$t \in \J$, we have,
\begin{equation*}
     \| \projPj \uhjlifte  - \uhjlifte   \|_a \le \ci   \Bigg(\sum_{p\in \J}\Bigg[ \sum_{t\in \J}  2 \|\ProjHj \up - \up \|_a \Bigg]^2 \Bigg)^{\frac{1}{2}}.
\end{equation*}
Then, we arrive at the inequality \eqref{ineq:norme-a-max} given by,
\begin{equation*}
     \| \projPj \uhjlifte  - \uhjlifte   \|_a \le \ci \max_{p\in \J} \|\up-\ProjHj \up\|_a.
\end{equation*}
To prove \eqref{ineq:U-PU-norme-a}, we need to keep in mind that the norms with respect to the bilinear forms~$a$ and~$\ahell$ are equivalent  and we use \eqref{ineq:norm(u-Hu)-Z-G} as follows, 
\begin{equation*}
    \|\up-\ProjHj \up\|_a  \le c \|\up-\ProjHj \up\|_{\ahell} 
     \le c \|\up-\ProgZj \up \|_{\ahell} + \ci\|\up-\projG \up \|_{\mhell}.
\end{equation*}
By applying again the norm equivalence and using the error estimations \eqref{err:eigenvectors-H1} and \eqref{ineq:Gellu-u-mhell}, we have,
\begin{equation*}
   \|\up-\ProjHj \up\|_a \le \ci ( h^k+h^{r+1/2}).
\end{equation*}
Combining the latter inequality with \eqref{ineq:norme-a-max}, we obtain \eqref{ineq:U-PU-norme-a}. \medskip

Passing to the proof of Inequality \eqref{ineq:norme-m-max}, we consider the norm with respect to $m$ of \eqref{eq:Plbdu-u} as follows,
\begin{equation*}
    \| \projPj \uhjlifte  - \uhjlifte   \|_m \le \sum_{p\in \J} |\alpha_p|  \Bigg[ \sum_{t\in \J}  |m(\ProjHj \up  -\up ,u_t)| \|u_t\|_m  +  \|\up-\ProjHj \up \|_m \Bigg]
\end{equation*}
Using Cauchy-Schwarz, we proceed in a similar manner as for the previous inequality,
\begin{align*}
    \| \projPj \uhjlifte  - \uhjlifte   \|_m & \le (\sum_{p\in \J} |\alpha_p|^2 )^{\frac{1}{2}} \Bigg(\sum_{p\in \J}  \Bigg[ \sum_{t\in \J}  |m(\ProjHj \up  -\up ,u_t)| \|u_t\|_m  +  \|\up-\ProjHj \up \|_m \Bigg]^2 \Bigg)^{\frac{1}{2}}\\
    & \le (C(N))^{\frac{1}{2}}  \Bigg(\sum_{p\in \J}  \Bigg[ \sum_{t\in \J}  \|\ProjHj \up -\up\|_m \|u_t\|_m \|u_t\|_m  +  \|\up-\ProjHj \up \|_m \Bigg]^2 \Bigg)^{\frac{1}{2}}\\
    & \le  \ci \Bigg(\sum_{p\in \J}\Bigg[ \sum_{t\in \J}  2 \|\ProjHj \up -\up\|_m \Bigg]^2 \Bigg)^{\frac{1}{2}},
\end{align*}
where we used $ \|u_t\|_m^2 = 1$. Consequently, we obtain the inequality \eqref{ineq:norme-m-max}. Lastly, using the error estimation \eqref{err:eigenvectors-L2}, we obtain \eqref{ineq:U-PU-norme-m}.
\end{proof}
\paragraph{Proof of Theorem \ref{th-error-bound-eigenfunction}: the eigenvalue estimation \eqref{err:eigenvalue}.}
\label{proof:eigenvalue-final}

Firstly we recall \eqref{lamda-lambdah},  we get,
$$
    |\lbdj-\lbdhj |  \le \|  \projPj U_j^\ell -U_j^\ell \|_a^2 + \lbdj \| \projPj U_j^\ell -U_j^\ell\|_m^2 +|a_h^\ell-a|(U_j^\ell,U_j^\ell) +\lbdhj |\mhell-m|(U_j^\ell,U_j^\ell). 
$$
Secondly, we use \eqref{ineq:U-PU-norme-a}, \eqref{ineq:U-PU-norme-m} and \eqref{ineq:m-mhell} to arrive at,
\begin{align*}
 |\lbdj-\lbdhj |  & \le \ci( h^{2k} +  h^{2r+1}) + \lbdj \ci ( h^{2k+2} + h^{2r+1}) + |a_h^\ell-a|(U_j^\ell,U_j^\ell) + c \lbdhj  h^{r+1} \|U_j^\ell\|^2_{\HH1\omgam} \\
 & \le \ci ( h^{2k} +  h^{2r+1}) + |a_h^\ell-a|(U_j^\ell,U_j^\ell) + c \lbdhj  h^{r+1} \|U_j^\ell\|^2_{\HH1\omgam}.
\end{align*}
The remaining term can be estimated as such by using \eqref{ineq:a-ahell},
$$
|a_h^\ell-a|(U_j^\ell,U_j^\ell)  \le c h^r \| \na U_j^\ell\|^2_{\L2(B_h^\ell)}+ c h^{r+1} \| U_j^\ell\|^2_{\HH1(\Gamma)}
$$
By adding and subtraction $\projPj U_j^\ell$ as follow, and then applying \eqref{ineq:U-PU-norme-a}, we get,
\begin{eqnarray*}
    |a_h^\ell-a|(U_j^\ell,U_j^\ell) 
    & \le & c h^r \| \na (\projPj U_j^\ell - U_j^\ell)\|^2_{\L2(B_h^\ell)} + c h^r \| \na (\projPj U_j^\ell)\|^2_{\L2(B_h^\ell)} +c h^{r+1}  \| U_j^\ell\|^2_{\HH1(\Gamma)} \\
    & \le & \ci h^r (h^{2k}+h^{2r+1}) + c h^r \| \na (\projPj U_j^\ell)\|^2_{\L2(B_h^\ell)} +c h^{r+1}  \| U_j^\ell\|^2_{\HH1(\Gamma)}\\
    & \le & \ci( h^{2k+r}+h^{3r+1}) + c h^r \| \na (\projPj U_j^\ell)\|^2_{\L2(B_h^\ell)} +c h^{r+1}  \| U_j^\ell\|^2_{\HH1(\Gamma)}.
\end{eqnarray*}
Since we have $\projPj U_j^\ell \in\F $ a linear combination of exacts eigenvalues, then~$\projPj U_j^\ell \in \H2 \omgam$ and the inequality \eqref{ineq:H1+h1/2} can be applied to it as follows,
\begin{align*}
 |a_h^\ell-a|(U_j^\ell,U_j^\ell) & \le \ci( h^{2k+r}+h^{3r+1}) + c h^{r} \big(h^{1/2} \|\projPj  U_j^\ell\|_{\H2 (\Omega)}\big)^2 +c h^{r+1}  \| U_j^\ell\|^2_{\HH1(\Gamma)} \\
 & \le  \ci( h^{2k+r}+h^{3r+1}) + c h^{r+1}  \|\projPj  U_j^\ell\|_{\H2 (\Omega)}^2 +c h^{r+1}  \| U_j^\ell\|^2_{\HH1(\Gamma)} \\
 & \le   \ci h^{r+1} ( \|\projPj  U_j^\ell\|_{\H2 (\Omega)}^2 +  \| U_j^\ell\|^2_{\HH1(\Gamma)}),
\end{align*}
where $\|\projPj  U_j^\ell\|_{\H2 (\Omega)}^2 +  \| U_j^\ell\|^2_{\HH1(\Gamma)}$ is uniformly bounded with respect to $h$ and $r$. Since the exact eigenfunctions are sufficiently regular and we supposed that $\|  U_j^\ell \|_{\L2 (\Omega)}=\|  u_t \|_{\L2 (\Omega)}=1$, by \eqref{exp:P_m-Ujell}, $\|\projPj  U_j^\ell\|_{\H2 (\Omega)}$ is bounded independently of h.  By Inequality \eqref{ineq:U-PU-norme-a}, $\dist(\uhjlifte, \F) \to 0$ where $\F$ is of finite dimension, we can bound $ \| U_j^\ell\|_{\HH1(\Gamma)}$ independently of $h$. Finally, replacing this inequality in the eigenvalue estimation, we get the desired result~\eqref{err:eigenvalue}.
%
%
%
\section{Numerical experiments}
\label{sec:numerical-ex}
In this section are presented numerical results 
aimed to illustrate the convergence estimates of Theorem \ref{th-error-bound-eigenfunction}. We perform these simulations in the two dimensional and three dimensional cases. The Ventcel problem \eqref{sys-eigenval-ventcel} is considered on various domains. 
The discrete problem~\eqref{fvd} is implemented and solved using the finite element library CUMIN~\cite{cumin}. The resolution of the spectral problem is done with the help of the library ARPACK\footnote{https://www.arpack.fr/}, which is a numerical software library for solving large scale eigenvalue problems. The symmetric case (the iterative Lanczos algorithm) is used in shift invert mode with a shift value~$\sigma=-1$ (in order to accurately compute the eigenvalues of smallest amplitude). For this method, linear
systems~$Ax = b$ have to be solved for a single matrix $A$ and for numerous varying right hand sides: a linear system solver is  required for the sparse CSR matrix $A$ that is symmetric and positive definite.

\medskip

In dimension $2$, the direct solver MUMPS\footnote{https://mumps-solver.org/index.php}(MUltifrontal Massively Parallel sparse direct Solver) is considered allowing fast computations. It is particularly well adapted in the present context where linear systems involving the same matrix $A$ have to be solved many times. Cholesky ${\rm LL^T}$ decomposition of a single (positive definite) CSR sparse matrix is computed once at the beginning and afterwards used for numerous linear equation resolutions all along the spectral Lanczos algorithm. The tolerance for the Lanczos algorithm was set to a low value $(1{\rm E{-12}})$: this allowed to compute quickly, while using MUMPS, the numerical errors up to error values of ${\rm 1 E{-11}}$, allowing us to study the convergence asymptotic regimes. More details on the computational efforts are given in the following paragraph devoted to the unit disk case. 

In dimension $3$, memory requirements imposed a lighter method: a conjugate gradient with Jacobi preconditioning has been used. The tolerances for the iterative algorithms (Lanczos and conjugate gradient) have been set to very low values $(1{\rm E{-14}})$: this generally allowed to compute accurately the numerical errors up to error values of ${\rm 1 E{-10}}$, which was necessary in order to well capture the convergence asymptotic regimes. The $3${\rm D} computations are the most demanding in terms of computational effort and time. Therefore, they deserved a specific attention, which is given in the paragraph dealing with the unit ball.

\medskip

Curved meshes of the domain $\Omega$ of geometrical order~$1\le r \le 3$ have been generated using the software Gmsh
\footnote{\url{https://gmsh.info/}}. All integral computations (either on the physical domain $\Omega$ or on the computational domain~$\Omega_h^r$) are performed on the reference simplex using changes of coordinates. These changes of coordinates are made on each element of the underlying mesh that is considered: either an affine mesh $\tauh$, a curved mesh $\taur$ or an exact mesh $\taue$. This allows to compute numerical errors such as $\|u^\ell_h - u\|_{{\rm L}^2(\Omega)}$ between the lift $u^\ell_h$ of a finite element function $u_h$ defined on $\Omega_h^{(r)}$ and a function $u$ defined on the smooth domain $\Omega$. On the reference simplex, high order quadrature methods are used such that the integration error is of lower order than the approximation errors that are evaluated in this section: it has systematically been verified that the integration errors have negligible influence over the forthcoming numerical results.

Convergence towards the eigenfunctions has only been studied on domains where the analytical solutions are known (the disk and the ball). On domains where the eigenfunctions are not analytically known, such a convergence study is much more complicated to handle. 
We would need to compute reference eigenfunctions on a refined reference grid and also to project the numerical solutions defined on coarser meshes. However, in the context of curved meshes, this would lead to non trivial difficulties, which cannot be considered in the present work.

\medskip

All numerical results presented in this section can be fully reproduced using dedicated source codes available on CUMIN Gitlab\footnote{Cumin GitLab deposit, \url{https://plmlab.math.cnrs.fr/cpierre1/cumin}}.
\subsection{The two dimensional case}

\paragraph{Eigenvalue estimate on a smooth domain.}

The Ventcel problem \eqref{sys-eigenval-ventcel} is considered on a smooth domain defined as the interior of a Jordan curve, denoted~$\gamma$. The curve $\gamma$ has been set in such a way to have a smooth and connex domain, which moreover is non-convex with no symmetries, in order to avoid eventual \textit{super convergence} properties. 
Indeed, the domain $\Omega$ is the interior of the Jordan curve $\gamma: \theta\in[0,2\pi]\to \gamma(\theta) \in \R^2$ satisfying $\gamma(0)=\gamma(2\pi)$. For any~$\theta \in [0, 2\pi],$ the function gamma is given by,
\begin{equation*}
\gamma(\theta) = (\kappa(\theta) \cos{\theta},\kappa(\theta) \sin{\theta}),
\end{equation*}
where $k(\theta) = 1 + \alpha  \cos \theta + \beta \sin \theta + \frac{\beta}{2} \sin 3\theta,$
with $\alpha=0.3$ and $\beta=0.4$.
Curved meshes of order~$r=1,\dots, 3$ are generated using CUMIN and Gmsh (see Figure~\ref{fig:flower} for linear and quadratic meshes). $\mathbb{P}^k$ finite element methods, with degrees~$k=1,\dots, 4$, are employed to estimate the eigenvalue error. 
\begin{figure}[H]
\centering
    \includegraphics[width=0.35\textwidth]{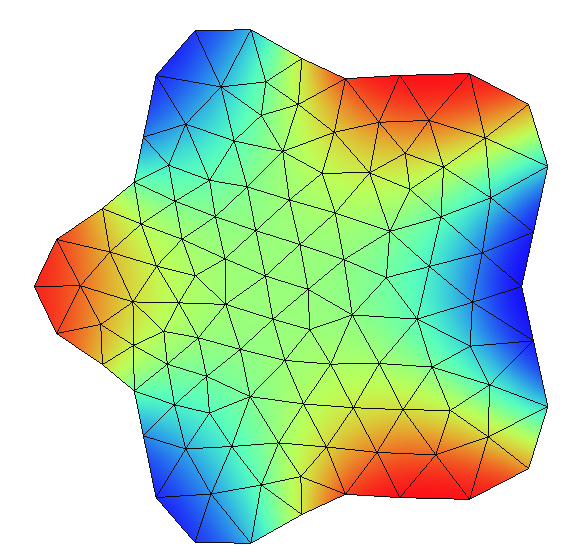}
    ~~~
    \includegraphics[width=0.35\textwidth]{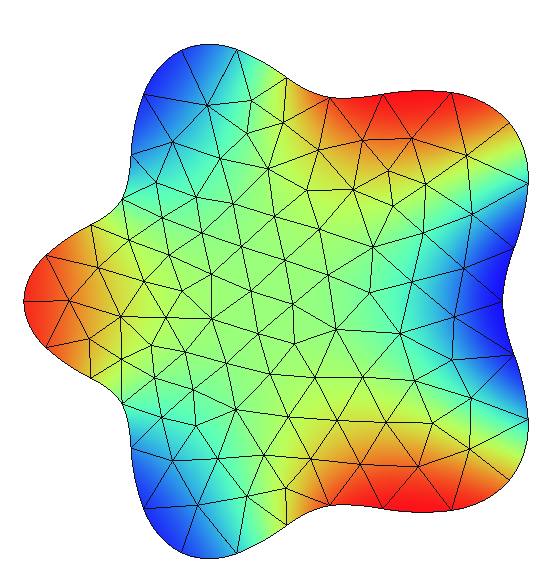}
 \caption{Representation of the $6^{\rm th}$ eigenfunction computed using $\P 3$ finite element on a (coarse) mesh of $\Omega$: affine mesh (left) and quadratic mesh (right).}
\label{fig:flower}
\end{figure}

The mesh degree and the finite element order being fixed, the 10 first eigenvalues are computed on a series of successively refined meshes: each mesh counts~$20\times 2^{n-1}$ edges on the domain boundary, for~$n=1,\dots, 5$. We do not know the exact eigenvalues of the Ventcel problem \eqref{sys-eigenval-ventcel} on this domain. Thus, reference eigenvalues have been computed on a reference mesh of order~$r=3$ using a~$\P 4$ finite element method. The reference mesh counts~$20 \times 2 ^5$ boundary edges and is made of approximately~$76\, 000$ cubic triangles, the associated~$\P 4$ finite element space has approximately~$610\, 000$~DOF (Degrees Of Freedom). We mention that the computation time is very fast in the present case: total computations roughly last one minute on a simple laptop, which are made really efficient with the direct solver MUMPS here.

\medskip

To calculate the eigenvalue error, we estimate the difference between the reference eigenvalues, denoted $\lbdj$, and the computed eigenvalues denoted $\lbdhj$. In Table \ref{TABLE2}, the convergence order of the error associated to the~$6^{\rm th}$ eigenvalue, given by $\elbd :=  |\lambda_6 - \Lambda_6|$, is presented. We mention that any other choice within the 10 eigenvalues that have been computed lead to the same convergence pattern. The convergence orders are evaluated from the error ratio between two successive meshes. The order estimations display very stable behaviour (no oscillation): we reported in Table~\ref{TABLE2} the convergence orders estimated between the two finest meshes.
%
%
%
%
\begin{table}[!ht]
    \centering
    \begin{tabular}{|l||l|l|l|l|}
\cline{2-5}
\multicolumn{1}{c||}{}    &  \multicolumn{4}{|c|}{\textbf{$\elbd$}}  \tabularnewline
\hline
Mesh type &   {$\P1$} &   {$\P2$} &   {$\P3$} &   {$\P4$}   \tabularnewline
\hline
 Affine \textcolor{ black}{(r=1)}   & 
 {1.96} & 2.00 &  {2.00} &  {2.00}  \tabularnewline
\hline
Quadratic \textcolor{ black}{(r=2)} & 
1.99 & 3.97 &  {3.98} &  {3.97} \tabularnewline
\hline
Cubic \textcolor{ black}{(r=3)} & 
1.99 & \textcolor{red}{2.99} & 4.07 &  \textcolor{ black}{4.08}\tabularnewline
\hline
\end{tabular}
 \caption{\label{TABLE2}Convergence order of $\elbd$ (Figures in red represent a loss in the convergence rate).}
\end{table}

As displayed in Table \ref{TABLE2}, the convergence rate of $\elbd$ on an affine mesh ($r=1$) are equal to~$r+1=2$ for any $\P k$ finite element method used as expected by the theory.

For the quadratic case ($r=2$), a \textit{super convergence} is observed: a saturation of the error occurs at order $4$ when it was expected to stop at $3$. This super convergence had already been observed in~\cite{Jaca}, \cite{art-joyce-1} and \cite{D4}: quadratic meshes seem to behave as if $r=3$, however no theoretical explanation of this phenomenon has been proposed so far to the authors' knowledge. It is interesting to notice that this super convergence also occurs in the present example though the domain is neither convex  nor symmetric. This implies that this phenomenon is not related to some particular geometric properties of the domain, as one might presume.

On the cubic meshes ($r=3$), the convergence order of $\elbd$ follows the expected estimate~\eqref{err:eigenvalue} and a saturation of the error is observed at order $r+1=4$. The only odd case worth mentioning is when using a $\P 2$ finite element method on a cubic mesh. In this particular case, we obtained a convergence order of 3 whereas the theory predicts a convergence order of 4. This loss is observed in all the numerical experiments throughout this work and it will be discussed in details in the following paragraph. 

\paragraph{Error estimates on the unit disk}

The Ventcel problem \eqref{sys-eigenval-ventcel} is considered on the unit disk~${\rm D(O,1)} \subset \R^2$. 
In this case, the eigenfunctions are the harmonic polynomials. A convergence analysis is performed on the $6^{\rm th}$ eigenvalue $\lambda_6$  of multiplicity $2$ with corresponding eigenspace, denoted ${\rm E_3}$, equal to the space of harmonic polynomials of degree $3$. 

To proceed, $\mathbb{P}^k$ finite element methods, of degrees $k=1,\dots, 4$, are used for the error estimates on meshes of order $r=1,\dots, 3$ (see Figure \ref{fig:disk-eigenfct} for linear and quadratic meshes). The mesh order and the finite element degree being fixed, the~$12$ first eigenvalues are computed on a series of five successively refined meshes: each mesh counts $20\times 2^{n-1}$ edges on the domain boundary, for~$n=1,\dots, 6$. On the most refined mesh using a $\P 4$ finite element method, we counted~$20 \times 2^5$~boundary edges and approximately $75\, 500$ triangles. The associated~$\P 4$ finite element space has approximately~$605\,600$~DOF. The computations are accomplished very quickly, the total computation time is less than four minutes on a regular computer.

\medskip

We denote~$\Lambda_6$ a numerical eigenvalue approximating $\lambda_6$ with $U_6$ as its associated computed eigenfunction. For each mesh order~$r$ and each finite element degree~$k$, 
the following numerical errors are computed on a series of refined meshes: 
$$
 \eL2 := \inf \{ \| U_6^\ell - u \|_{\L2 (\Omega)}, u\in {\rm E_3} \}, \quad 
 \eH1 := \inf \{ \| \na (U_6^\ell - u ) \|_{\L2 (\Omega)}, u\in {\rm E_3} \}, 
$$
$$
  \mbox{ and } \quad  \elbd :=  |\lambda_6 - \Lambda_6|.
$$
The $\L2$ distance between $U_6^\ell$ and the eigenspace ${\rm E_3}$, denoted $\eL2$, is computed using the $\L2$ orthogonal projection of $U_6^\ell$ onto ${\rm E_3}$. In a similar manner, the $\L2$ distance between $\nabla U_6^\ell$ and the space~$\na {\rm E_3} = \{\na u, u \in {\rm E_3}\}$, denoted $\eH1$, is also computed using the $\L2$ orthogonal projection of~$\na U_6^\ell$ onto $\na {\rm E_3}$. 

In Tables~\ref{tab:eigenfunction-disk} and \ref{TABLE:eigenval-disk}, the convergence orders of $\eL2$, $\eH1$ and $\elbd$ are reported. They are evaluated from the error ratio between two successive meshes that display very stable behaviour, detecting no oscillation. The displayed error rates are estimated between the two finest meshes.
\begin{figure}[H]
\centering
\includegraphics[width=0.25\textwidth]{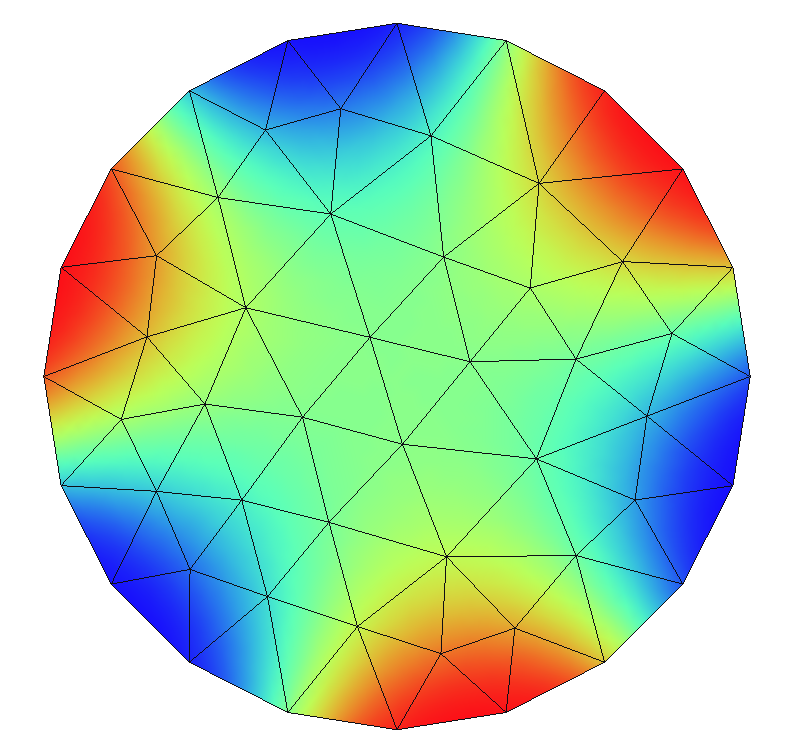}
~~~
    \includegraphics[width=0.25\textwidth]{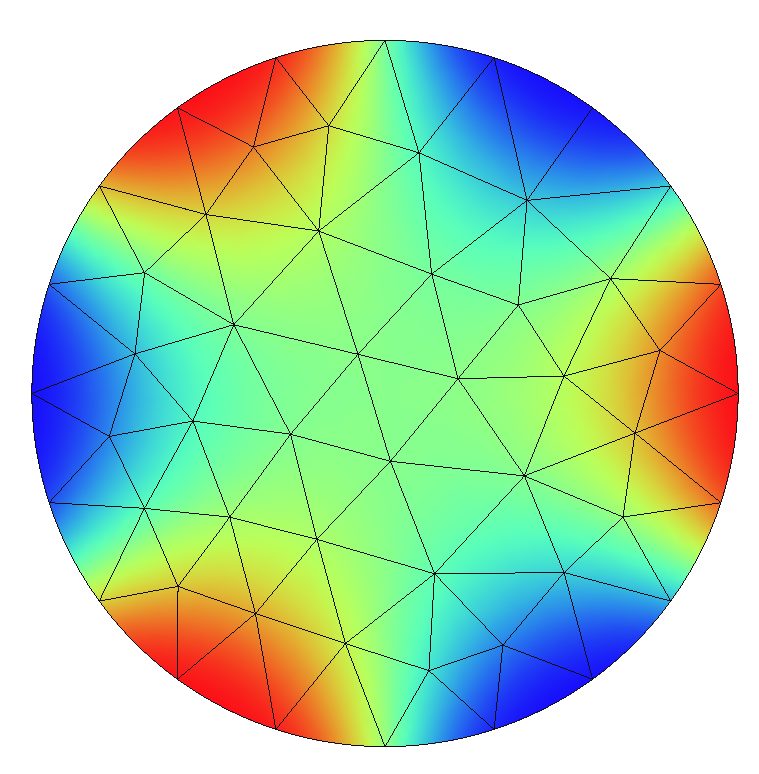}
 \caption{Display of the eigenfunction $U_6$ associated to the computed eigenvalue $\Lambda_6$ using $\P 3$ method on an affine mesh (left) and a quadratic mesh (right).}
\label{fig:disk-eigenfct}
\end{figure}
\begin{table}[!ht]
    \centering
    \begin{tabular}{|l||l|l|l|l||l|l|l|l|}
\cline{2-9}
\multicolumn{1}{c||}{}    &  \multicolumn{4}{|c||}{$\eL2$} & \multicolumn{4}{|c|}{$\eH1$}  \\[0.08cm]
\cline{2-9}
\multicolumn{1}{c||}{}    &  $\P1$ &  $\P2$ &  $\P3$ &  $\P4$ & $\P1$ &  $\P2$ &  $\P3$ &  $\P4$  \tabularnewline
\hline
 Affine mesh (r=1)   & 
 {2.01} & {2.48}  & 2.48 & 2.48 & 
1.00 & 1.51 &  1.50 & 1.50 \tabularnewline
\hline
Quadratic mesh (r=2) & 
2.01 & 3.07 &  {4.5} & {4.47} & 1.00 & 2.01 & {3.5} & 3.49 \tabularnewline
\hline
Cubic mesh (r=3) & 
2.01 & \textcolor{red}{2.47}& {3.48} & {4.49} & 
0.99 &\textcolor{red}{1.49} & \textcolor{red}{2.48} & 3.49 \tabularnewline
\hline
\end{tabular}
\caption{Convergence order of the eigenfunctions errors in $\L2$ and $\HH1_0$ norms (Figures in red represent a loss in the convergence rate).
  \label{tab:eigenfunction-disk}
}
\end{table}
The $\HH1_0$ error convergence rate in Table \ref{tab:eigenfunction-disk} is equal to~$\min \{ k, r+1/2 \}$, for the most part: on an affine mesh, the order of $\eH1$ is equal to $1.5$, for all $\P k$ method with $k \ge 2$, as expected. On the quadratic mesh, similarly to the result in Table \ref{TABLE2},  the quadratic mesh acts like a cubic mesh: the error rate is equal to $3.5$ instead of $2.5$ for a $\P4$ method, as if r is equal to $3$. However, one needs to point out that, with a $\P 3$ method, the order is equal to $3.5$ surpassing the expected value equal to $3$. A possible explanation for this phenomenon is that the eigenspace ${\rm E_3}$ associated to $\lambda_6$ is equal to the space of harmonic polynomials of degree $3$ on the disk, as stated before. Moreover, the finite element approximation space $\Vh$ is also made of polynomials on most of the domain (all the elements that do not have an edge on the boundary, i.e. $\Omega\setminus B_h^\ell$ where $B_h^\ell$ is defined in Corollary~\ref{coro:h-1/2}). This large vicinity between ${\rm E_3}$ and $\Vh$ may be a possible cause for the super convergence observed here. Lastly, on the cubic mesh, for a $\P 4$ method the rate of $\eH1$ is equal to~$3.5$, as anticipated. However, oddly, for a $\P 2$ and $\P 3$ method, a loss in the order of convergence of $\eH1$ is depicted and highlighted in red. Instead of having a convergence rate equal to $2$ (resp.~$3$) for a $\P 2$ (resp. $\P 3$) method we obtained $1.5$ (resp. $2.49$). This is discussed in more details in the following paragraph.

The $\L2$ error convergence rates are displayed in Table \ref{tab:eigenfunction-disk}, where a super convergence is quickly noticed: in the affine case $(r=1)$, the convergence rate of~$\eL2$ is equal to $2.5$ instead of~$1.5$ for a~$\P k$~method with $k\ge 2$. As discussed previously, a super convergence is observed on the quadratic mesh: the quadratic mesh acts like the cubic mesh with $r=3$. However, the convergence order depicted in Table \ref{tab:eigenfunction-disk} is equal to $4.5$ surpassing the expected order of $3.5$ for a $\P 3$ and $\P 4$ method. In the cubic case, the convergence rate is equal to $4.5$ instead of $3.5$ for a $\P 4$ method. A plausible reason for this super convergence on all curved meshes of order $r=1,\, 2,\, 3$ is that the $\L2$ estimate~\eqref{err:eigenvectors-L2} is not optimal. 
In the light of these numerical results, we can formulate the following conjecture which may be a more accurate version of the obtained estimate~\eqref{err:eigenvectors-L2}:
\begin{equation}
    \label{ineq:L2-modif} \eL2 \le \ci ( h^{k+1} + h^{r+1}).
\end{equation}
We obtained a similar error estimate in the $\L2$ norm in~\cite{art-joyce-1} for a Ventcel problem with source terms.
However we have not been able to prove this estimate  \eqref{ineq:L2-modif}.
One has to point out that even with the estimate \eqref{ineq:L2-modif} a super convergence is still observed in the following cases: on affine mesh with a $\P k$ method with $k \ge 2$, the rate of $\eL2$ is equal to $2.5$ instead of $2=r+1$. Additionally, on quadratic meshes with a $\P 3$ and $\P 4$ method, the order of $\eL2$ is equal to $4.5$ instead of $4=r+1$. Similarly, on cubic meshes with a $\P 4$ method, the error order is equal to $4.5$ instead of $4=r+1$. As stated in the case of the~$\HH1_0$ error, the large similarity between the eigenspace ${\rm E_3}$ associated to~$\lambda_6$ and the finite element space $\Vh$ may be a possible cause for this super convergence. 

\medskip

One needs to stress that for these both errors $\eH1$ and $\eL2$, similarly to the results in Table~\ref{TABLE2}, a loss in the convergence rate is detected on a cubic mesh with a~$\P 2$ and $\P 3$ method. This convergence default of $-1/2$ is already discussed for the $\HH1_0$ error in the previous paragraph, and it is also observed in the case of the $\L2$ error: the convergence rate of $\eL2$ is equal to $2.5$  for the~$\P 2$ method instead of $3$. So far we are not able to fully explain this convergence default on cubic meshes. Moreover as discussed in \cite{art-joyce-1}, we noticed that it is only related to "volume norms", since the numerical errors computed in $\L2(\Gamma)$ and $\HH1(\Gamma)$ norms show the expected convergence rate. Numerical experiments we have led in this direction show that this lack of convergence is not related to the lift. This lack of convergence is associated with the interpolation error  between a smooth function $u$ and its finite element interpolent $\mathcal{I} u\in \Vh$, denoted $\| \mathcal{I}  u - u\|_{\HH1(\omhr)}$. On the considered cubic meshes, this error behaves like $O(h^{k-1/2})$ instead of $O(h^k)$ for $k \ge 2$. While conducting some experiments, we noticed that this interpolation error is highly sensitive to the position of the central node in cubic elements without being able so far to overcome this issue.
%
%
%
%
%
%
\begin{table}[!ht]
    \centering
    \begin{tabular}{|l||l|l|l|l|}
\cline{2-5}
\multicolumn{1}{c||}{}    &  \multicolumn{4}{|c|}{\textbf{$\elbd$}}  \tabularnewline
\hline
Mesh type &   {$\P1$} &   {$\P2$} &   {$\P3$} &   {$\P4$}   \tabularnewline
\hline
 Affine \textcolor{ black}{(r=1)}   & 
2.00 & 2.00 & 2.00 & 2.00  \tabularnewline
\hline
Quadratic \textcolor{ black}{(r=2)} & 
2.00 & 4.01 & 4.01 &  {3.99} \tabularnewline
\hline
Cubic \textcolor{ black}{(r=3)} & 
2.00 &  \textcolor{red}{3.27} & {3.89} &  4.00\tabularnewline
\hline
\end{tabular}
 \caption{\label{TABLE:eigenval-disk}Convergence order of $\elbd=|\lambda_6-\Lambda_6|$ (Figures in red represent a loss in the convergence rate).}
\end{table}

The convergence rates of $\elbd$ observed in Table~\ref{TABLE:eigenval-disk} are analogous to the results of Table \ref{TABLE2}. As anticipated, the quadratic mesh $(r=2)$ behaves as if $r$ is taken equal to $3$: the convergence rate is equal to $4$ instead of $3$, for all $\P k$ method with~$k \ge 2$. A loss in the convergence rate is highlighted in red in Table~\ref{TABLE:eigenval-disk}  and in Table \ref{TABLE2}, in the cubic case $(r=3)$ for a $\P 2$ method. Indeed, in the same case, the~$\HH1_0$ order of convergence for the associated eigenfunction in Table \ref{tab:eigenfunction-disk} is equal to $1.5$ instead of $2$ (see Table \ref{tab:eigenfunction-disk}). This seems to imply an order of convergence of $2\times1.5=3$ instead of~$2\times2=4$ for the eigenvalues. 
\subsection{A 3D case: error estimates on the unit ball}

To conclude these numerical experiments, the system \eqref{sys-eigenval-ventcel} is now considered on the unit ball~${\rm B(O,1)} \subset \R^3$. The ball is discretized using meshes of order~$r=1,\dots, 3$, which are depicted in Figure~\ref{f1} for affine and quadratic meshes. A convergence analysis is performed on the~$10^{\rm th}$~eigenvalue~$\lambda_{10}$ of multiplicity~$7$. Since on the unit ball, the eigenfunctions are the harmonic polynomial, the corresponding eigenspace ${\rm E_3}$ to $\lambda_{10}$ is equal to the space of harmonic polynomials of degree $3$. 

For each mesh order $r$ and finite element degree $k$, we compute the~$12$ first eigenvalues on a series of five successively refined meshes: it has been necessary to consider these five meshes in order to obtain a reliable estimation of the convergence rates (considering a ${\rm 6^{th}}$ mesh however would have been unaffordable in terms of computational efforts). Each mesh counts~$20\times 2^{n-1}$ edges on the equator circle, for~$n=1, \dots ,5$. The most refined mesh has approximately $2,4 \times 10^6$ tetrahedra and the associated $\P 3$ finite element method counts  $11 \times 10^6 $ degrees of freedom. Consequently the matricial system of the spectral problem, which needs to be solved, has a size~$11 \times 10^6 $ with a rather large stencil. As a result, in the 3D case, the computations are much more demanding, both in terms of CPU time and of memory consumption. The use of MUMPS, as we did in the 2D case, is no longer an option due to memory limitation.  The inversion of the linear system is done using the conjugate gradient method with a Jacobi preconditioner. With this strategy,~8 iterations of ARPACK were in general required to reach convergence (with a tolerance threshold of ${\rm 1E-14}$ as stated in this section's introduction): each iteration of ARPACK required roughly $130$ linear system inversions, each of which involving $2\, 000$ iterations of the preconditioned CG algorithm. To handle these computations, we resorted to the UPPA research computer cluster PYRENE\footnote{PYRENE Mesocentre de Calcul Intensif Aquitain, https://git.univ-pau.fr/num-as/pyrene-cluster}. Using shared memory parallelism on a single CPU with $32$ cores and $2\, 000$ Mb of memory, each case required between $10$ to $30$ hours of computations. 

\begin{figure}[H]
\centering
\includegraphics[width=0.25\textwidth]{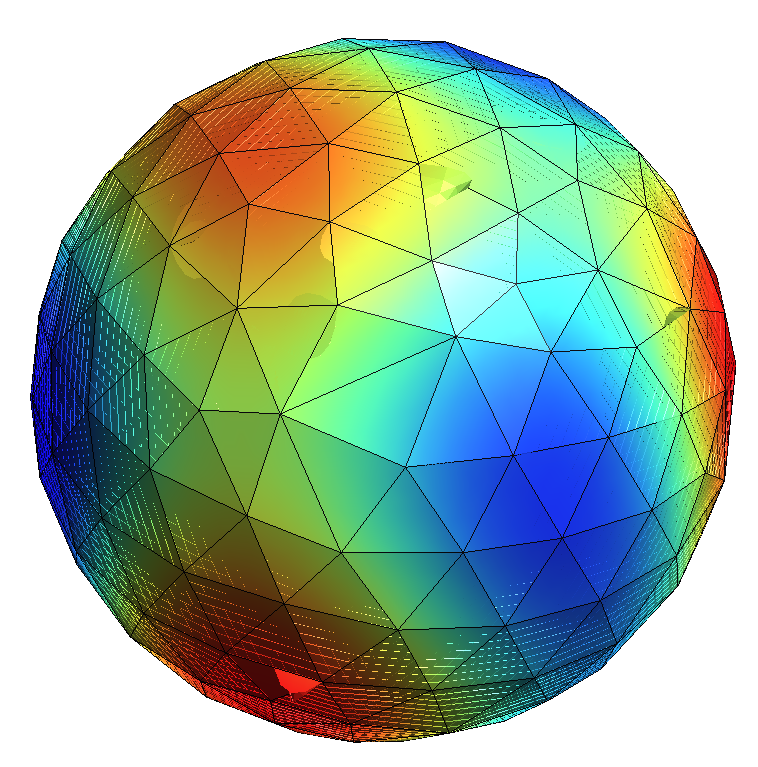}
~~~~~~
    \includegraphics[width=0.25\textwidth]{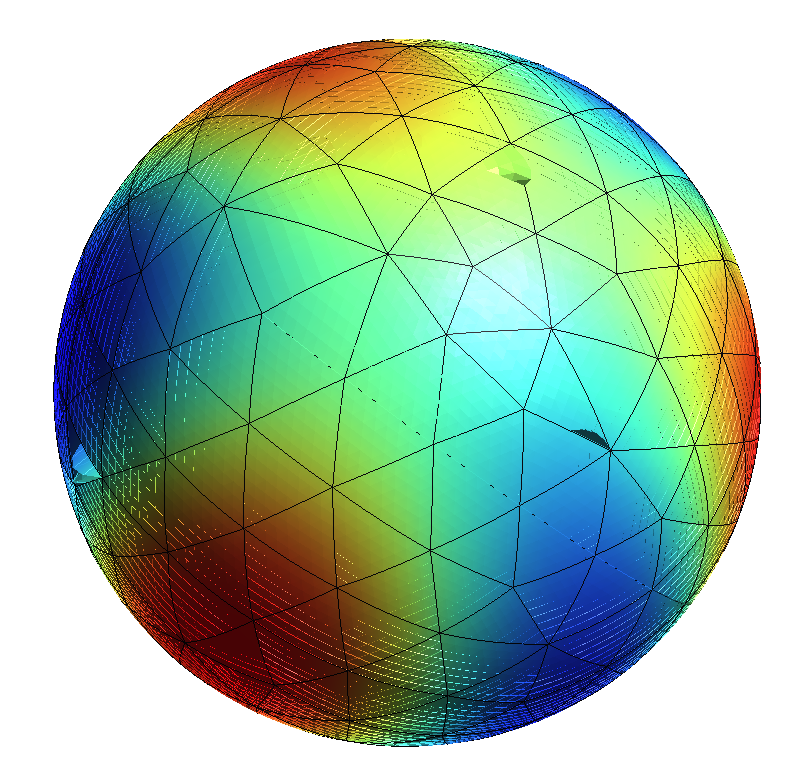}
 \caption{Display of the eigenfunction associated with the eigenvalue $\Lambda_{10}$ using $\P 2$ finite element on an affine mesh (left) and a quadratic mesh (right).}
\label{f1}
\end{figure}

\medskip

Denote~$\Lambda_{10}$ a numerical eigenvalue approximating $\lambda_{10}$ with $U_{10}$ as its associated computed eigenfunction. In each case, the following numerical errors are computed on a series of refined meshes, 
$$
 \eL2 := \inf \{ \| U_{10}^\ell - u \|_{\L2 (\Omega)}, u\in {\rm E_3} \}, \quad 
 \eH1 := \inf \{ \| \na (U_{10}^\ell - u ) \|_{\L2 (\Omega)}, u\in {\rm E_3} \}, 
$$
$$
  \mbox{ and } \quad  \elbdd :=  |\lambda_{10} - \Lambda_{10}|.
$$
Similarly to the disk case, orthogonal projections onto ${\rm E_3}$ are used in order to compute the $\L2$ (resp. $\HH1_0$) distance between $U_{10}^\ell$ and the eigenspace ${\rm E_3}$, denoted~$\eL2$ (resp. $\eH1$). 

\begin{figure}[H]
\centering
\includegraphics[width=0.45\textwidth]{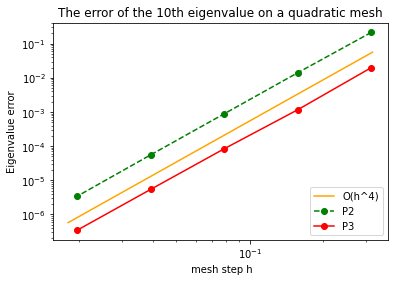}
\includegraphics[width=0.44\textwidth]{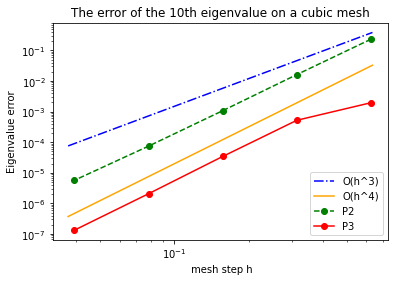}
 \caption{Display of the convergence rate of $\elbdd=|\lambda_{10}-\Lambda_{10}|$ using $\P 2$ and $\P 3$ finite element on a quadratic mesh (left) and a cubic mesh (right).}
\label{Fig:courbes-lambda_10}
\end{figure}
In Figure \ref{Fig:courbes-lambda_10}, is displayed a log–log graph of the error $\elbdd$ with respect to the mesh step, on a quadratic mesh (right) and a cubic mesh (left). In the quadratic case, as in the two dimensional experiments, the error is in $O(h^4)$ whereas $O(h^3)$ was expected from the theory. The same super convergence of quadratic meshes is present as in the 2D case: it is very interesting to underline this behaviour of the quadratic meshes, which brought a $O(h^4)$  geometric error also in three dimensions.
In the cubic case, when using a $\P 2$ method, the convergence rate of~$\elbdd$ starts around 4 tending to $3$, in hopes of following the loss in the convergence rate observed in the 2D case. {Note that for the ${\rm 10^{th}}$ eigenvalue, its asymptotic regime is quite harder to capture than the first ones. For an eigenvalue $\lambda_j$ with lower rank $j<10$, the convergence rate goes faster to 3, strengthening the hypothesis of a convergence loss in this case, as observed in the 2D case.} Finally, when using a $\P 3$ method on a cubic mesh the error $\elbdd$ seems to  be in $O(h^4)$, following Inequality \eqref{err:eigenvalue}.
\begin{figure}[H]
\centering
\includegraphics[width=0.45\textwidth]{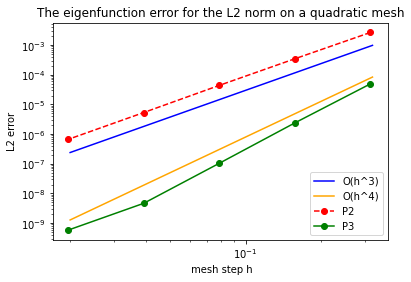}
\includegraphics[width=0.44\textwidth]{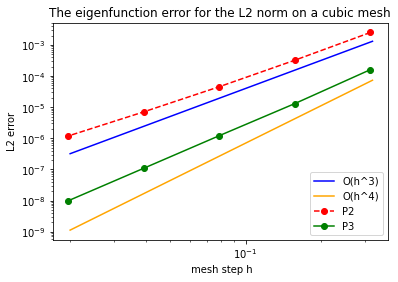}
 \caption{Display of the convergence rate of $\eL2$ using $\P 2$ and $\P 3$ finite element on a quadratic mesh (left) and a cubic mesh (right).}
\label{Fig:courbes-L2}
\end{figure}

In Figure \ref{Fig:courbes-L2}, is displayed a log–log graph  of the $\L2$ error $\eL2$ with respect to the mesh step, on a quadratic mesh (right) and a cubic mesh (left).  On the quadratic mesh, for a $\P 2$ method, the order of $\eL2$ is equal to $3$, as expected. However, for a method of degree $k=3$, the error order seems to be slightly more than $4$ for the first $4$ meshes. Though the convergence rate decreases on the last point, this seems to confirm that super convergence for quadratic meshes also holds on the eigenfunctions in 3D. In the cubic case with a $\P 2$ (resp. $\P 3$) method, the graph of $\eL2$ seems to have a slope approximately equal to $2.5$ (resp. $3.5$). The same loss in convergence as in the~2D case is observed, see Table \ref{tab:eigenfunction-disk}: {this convergence default of $-1/2$ has been formerly discussed in the~2D case.} 
\begin{figure}[H]
\centering
\includegraphics[width=0.45\textwidth]{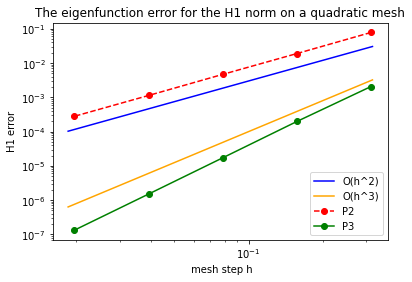}
\includegraphics[width=0.44\textwidth]{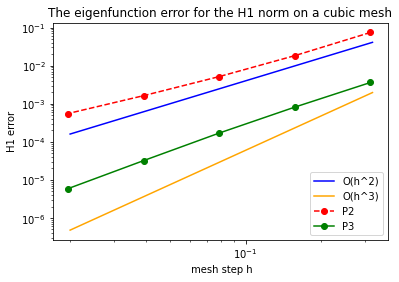}
 \caption{Display of the convergence rate of $\eH1$ using $\P 2$ and $\P 3$ finite element on a quadratic mesh (left) and a cubic mesh (right).}
\label{Fig:courbes-H1}
\end{figure}

Concerning the $\HH1_0$ error, the results obtained in Figure \ref{Fig:courbes-H1} agree with the rates obtained on the disk. The convergence order is equal to~2 on the quadratic mesh with a finite element degree $k=2$. With a $\P 3$ method, the graph of $\eH1$ seems to have a sloop around $3.5$ higher than the awaited value of $3$. As for the cubic mesh, the loss in the convergence rate of $\eH1$ was already observed and discussed in the case of the disk: for a $\P 2$ (resp. $\P 3$) method, one can assess that the order of $\eH1$ is slightly less than $2$ (resp. $3$), similarly to Table \ref{tab:eigenfunction-disk}.
%
%
%
%

\appendix

\section{Mesh construction}
\label{mesh:appendix}
%
%
%
%
%
%
%
\subsection{Affine mesh $\tauh$} 
Let $\tauh$ be a polyhedral mesh of $\Omega$ made of simplices of dimension $d$ (triangles or tetrahedra), it is chosen as quasi-uniform and henceforth shape-regular (see  \cite[definition 4.4.13]{quasi-unif}). 
Define the mesh size~$h:= \max\{\mathrm{diam}(T); T \in \tauh \}$, where $\mathrm{diam}(T)$ is the diameter of $T$. 
The mesh domain is denoted by $\omh1:= \cup_{T\in  \tauh}T$. Its boundary denoted by $\gh1 :=\partial \omh1$ is composed of $(d-1)$-dimensional simplices that form a  mesh of $\Gamma = \partial \Omega$. The vertices of $\gh1$ are assumed to lie on $\Gamma$. 
%
%
%
%
%
%
%
%
%
For $T \in \tauh$, we define an affine function that maps the reference element $\tref$ onto~$T$, 
$
\ft : \tref \to T:=\ft(\tref).
$
For more details, see \cite[page 239]{ciaravtransf}. 
%
%
%
%
%
%
%
%
%
\subsection{Exact mesh $\taue$}
In the 1970's, many authors gave an explicit construction of an exact triangulation (see  \cite{scott,Lenoir1986}). In this subsection, is recalled the definition of an exact mesh given in \cite[\S 4]{elliott}, \cite[\S 3.2]{ed} and \cite[\S 2]{Jaca}. The present definition of an exact transformation~$\fte$ combines the definitions found in \cite{Lenoir1986,scott} with the orthogonal projection $b$ as used in \cite{dubois}. 

\medskip

Let us first point out that for a sufficiently small mesh size $h$, a mesh element~$T$ cannot have $d+1$ vertices on the boundary $\Gamma$, due to the quasi uniform assumption imposed on the mesh $\tauh$. 

\begin{definition}
\label{appdef:sigma-lambdaetoile-haty}
    Let $T\in\tauh$ be a non-internal element (having at least 2 vertices on the boundary). Denote~$v_i = \ft(\hatv_i)$ as its vertices, where $\hatv_i$ are the vertices of~$\tref$. We define $\varepsilon_i=1$  if $v_i\in  \Gamma$ and~$\varepsilon_i=0$ otherwise.
    To $ \hatx\in \tref$ is associated its barycentric coordinates $\lambda_i$ associated to the vertices $\hatv_i$ of~$\tref$ and $\lambdaetoile (\hatx):= \sum_{i=1}^{d+1} \varepsilon_i \lambda_i$ (shortly denoted by $\lambdaetoile$). Finally, we define  $\hat{\sigma} : = \{ \hatx \in \tref; \lambdaetoile(\hatx) = 0 \}$ and the function $
    \haty := \dfrac{1}{\lambdaetoile}\sum_{i=1}^{d+1} \varepsilon_i \lambda_i\hatv_i\in\tref$, which is well defined on $\trefminissigma$.
\end{definition}

Consider a non-internal mesh element $T \in \tauh$ and the affine transformation $\ft$. In the two dimensional case, $\ft(\hatsigma)$ will consist of the only vertex of $T$ that is not on the boundary $\Gamma$. 
In the three dimensional case, the tetrahedral $T$ either has 2 or 3 vertices on the boundary. In the first case, $\ft(\hatsigma)$ is the edge of $T$ joining its two internal vertices. In the second case, $\ft(\hatsigma)$ is the only internal vertex of $T$. 
\begin{definition}
\label{def:fte-y}
We denote~$\taue$ the mesh consisting of all exact elements $\te=\fte(\tref)$, where~$\fte = \ft$ for all internal elements, as for the case of non-internal elements $\fte$ is given by, 
\begin{equation}
  \label{eq:def-fte}
\fonction{\fte}{\tref  }{\te :=\fte( \tref) }{  \hatx}{\displaystyle  \fte( \hatx) := \left\lbrace 
\begin{array}{ll}
 x & {\rm if } \, \hatx \in \hatsigma, \\
     x+(\lambdaetoile)^{r+2} ( b(y) - y) &  {\rm if } \, \hatx \in \trefminissigma ,
\end{array}
\right.}
\end{equation}
with $x = \ft( \hatx)$ and $y = \ft( \haty)$. 
It has been proveen in \cite{elliott} that $\fte$ is a $\c1$-diffeomorphism and~$C^{r+1}$ regular on $\tref$.  
\end{definition}
\section{The lift transformation}
\label{lift}
We recall that the idea of lifting a function from the discrete domain onto the continuous one was already treated and discussed in many articles dating back to the 1970's, like \cite{nedelec,scott,Lenoir1986,Bernardi1989}. The key ingredient is a well defined lift transformation going from the mesh domain onto the physical domain $\Omega$.

We recall the \textit{lift} transformation $\Ghr$, which was defined in \cite[\S 4]{art-joyce-1}. Following the notations given in Definition \ref{appdef:sigma-lambdaetoile-haty}, the transformation~$\Ghr$, is given piecewise for all $\tr\in \taur$ by,
\begin{equation}
  \label{eq:def-fter}
  {\Ghr}_{|_{\tr}} \!\! := \ftre \circ ({\ftr})^{-1},
  \;
  \ftre( \hatx) \!:=\! \left\lbrace 
\begin{array}{ll}
 \!\!   x & {\rm if } \, \hatx \in \hatsigma \\
  \!\!   x+(\lambdaetoile)^{r+2} ( b(y) - y) &  {\rm if } \, \hatx \in \trefminissigma ,
\end{array}
\right. 
\end{equation}
with $ x := \ftr( \hatx)$ and $y := \ftr( \haty)$, where $\ftr$ is the $\P r$-Lagrange interpolant of $\fte$ defined Section \ref{sec:mesh}. By definition, we have ${\Ghr}_{|_{\tr}} = id_{|_{\tr}}$, for any internal mesh element $\tr~\in~\taur$.

\section{Proof of inequality~\eqref{ineq:Gellu-u-mhell}}%
\label{appendix:proof-mhell-u-Gu}
Keeping in mind the definition of $\projG : \HH1 \omgam \to \Vh$ as the Riesz projection in Definition \ref{def:projections}, we want to prove the estimate \eqref{ineq:Gellu-u-mhell}, given by, 
\begin{equation*}
    \| u - \projG u \|_{\mhell} \le ch^{k+1}.
\end{equation*}

To prove this estimate, we need to recall the definition of an interpolant, which can be found in~\cite{art-joyce-1}. Keeping in mind that $\Vh$ denotes the $\P k$-Lagrangian finite element space, let the $\P r$-Lagrangian interpolation operator be denoted by $\Ihr: \c0 (\omhh) \to \Vh$. 

The lifted finite element space is given by,
$
     \Vhlifte := \{ v_h^\ell, \ v_h \in \Vh \},
$
with its lifted finite element interpolation operator $\Ihlifte$ defined as follows, 
\begin{equation*}
    \fonction{\Ihlifte}{\c0 (\Omega)}{\Vhlifte}{v}{\Ihlifte (v) := \big( \Ihr (v \circ \Ghr) \big)^\ell.} 
\end{equation*}
Notice that, since $\Omega$ is an open subset of $\R^2$ or $\R^3$, then we have the following Sobolev injection~$\Hexpo{k+1}(\Omega) \hookrightarrow \c0 (\Omega)$. 
Thus, any function $w \in \Hexpo{k+1}(\Omega)$ may be associated to an interpolation element~$\Ihlifte(w) \in \Vhlifte$. 

We recall its associated interpolation inequality, which is proved in \cite[Proposition~6]{art-joyce-1}.
\begin{proposition}
Let $v \in \Hk1 \omgam$ and $2 \le m \le k+1$. The operator $\Ihlifte$ satisfies the interpolation inequality with a constant $c>0$ as follows,
\begin{equation}
\label{interpolation-ineq}
    \|v-\Ihlifte v\|_{\L2\omgam} + h \|v-\Ihlifte v\|_{\HH1\omgam} \le c h^{m} \|v\|_{\Hexpo{m} \omgam}. 
\end{equation}
\end{proposition}

Secondly, we define the functional~$F_h$ on $\HH1 \omgam$ as follow,
$$
\fonction{F_h }{\HH1 \omgam  }{\R}{v}{\displaystyle F_h(v)= (a-\ahell)(u-\projG u, v),} 
$$
where a is the continuous bilinear forms defined on $[\HH1\omgam]^2$ in Section \ref{sec:notations_def} and $\ahell$ is the lift of the discrete bilinear form defined on $[\Vhlifte]^2$ in Section \ref{sec:FEM}.
Notice that for $v \in \Vhlifte,$ $F_h(v)=a(u-\projG u, v)$, by Definition \ref{def:projections} of $\projG$.

In order to prove the inequality \eqref{ineq:Gellu-u-mhell}, we proceed by bounding $F_h$ as follows in Lemma \ref{lem:F-inequalities}, with the help of the inequality \eqref{interpolation-ineq}
\begin{lem}
\label{lem:F-inequalities}
    Let $v\in \HH1 \omgam$, there exists $c>0$ such that,
    \begin{equation}
    \label{ineq:F_h(1)}
        |F_h(v)|  \le c h^{k+r} \|v\|_{\HH1 \omgam}.
    \end{equation}
\end{lem}
\begin{proof}
 Denote $e:=u-\projG u.$ Let $v\in \HH1 \omgam,$ using Inequality \eqref{ineq:a-ahell} we have,
\begin{multline*}
    |F_h(v)|  = |a-a_h^\ell|(e,v)  \le ch^r \|e\|_{\HH1 \omgam} \|v\|_{\HH1 \omgam} + ch^{r+1} \|e\|_{\HH1 \omgam} \|v\|_{\HH1 \omgam} \\
     \le ch^r (\|e\|_{\HH1 \omgam} + ch \|e\|_{\HH1 \omgam} ) \|v\|_{\HH1 \omgam}.
\end{multline*}
Then applying the $\HH1$ error inequality \eqref{ineq:Gellu-u-ahell}, we get, 
$$
    |F_h(v)|  \le ch^r (h^k + h^{k+1}) \|v\|_{\HH1 \omgam}  \le c h^{k+r} \|v\|_{\HH1 \omgam}.
$$
\end{proof}
\begin{proof}[Proof of Inequality \eqref{ineq:Gellu-u-mhell}]
To prove estimate \eqref{ineq:Gellu-u-mhell}, we use an Aubin-Nitche argument. Since $e \in \L2 (\Omega),$ then there exists a unique solution $z_e \in \H2 \omgam$ solution of the weak formulation~\eqref{fv_faible} with source terms satisfying,
\begin{equation}
\label{ineq:z_e-e}
    \|z_e\|_{\H2 \omgam} \le  c\| e\|_{\L2 (\Omega)}.
\end{equation}
We have, using the continuity of the bilinear form $a$,
\begin{multline*}
    \|u-\projG u \|^2_{\L2 (\Omega)}  =  \| e \|^2_{\L2 (\Omega)}  = a(e,z_e) 
     = a(e,z_e-\Ihlifte z_e) +a(e,\Ihlifte z_e)\\
      \le c_{cont} \|e\|_{\HH1 \omgam} \|z_e - \Ihlifte z_e\|_{\HH1 \omgam} + |F_h(\Ihlifte z_e)|.
\end{multline*} 
We apply the inequalities \eqref{ineq:Gellu-u-ahell}, \eqref{interpolation-ineq} for $z_e\in \H2\omgam$ and \eqref{ineq:F_h(1)} since $\Ihlifte z_e \in \Vhlifte$, as follows,
\begin{align*}
    \|u-\projG u \|^2_{\L2 (\Omega)} 
     & \le c (h^{k}) h \|z_e \|_{\H2 \omgam} + ch^{k+r} \|z_e\|_{\HH1 \omgam} \le c h^{k+1} \|z_e\|_{\H2\omgam} .
\end{align*} 
By applying  \eqref{ineq:z_e-e} and dividing by $  \|u-\projG u \|_{\L2 (\Omega)} $, we obtain,
\begin{equation*}
    \|u-\projG u \|_{\L2 (\Omega)} 
\le c h^{k+1} .
\end{equation*}
By the equivalence between the norms the norms $ \|\cdot\|_{m}=\|\cdot\|_{\L2(\Omega)}$ and $\|\cdot\|_{\mhell}$ in Corollary \ref{coro:norm-equiv-with-h-r}, we obtain \eqref{ineq:Gellu-u-mhell}.

\end{proof}

\bibliographystyle{abbrv}
\bibliography{biblio}

\end{document}